\documentclass[amsthm]{elsart}
\usepackage{amsmath}
\usepackage{amssymb}
\usepackage{amsfonts}
\usepackage{graphicx}
\usepackage{pb-diagram}
\usepackage{lineno}
\linenumbers

\newcommand{\eps}{\varepsilon}
\newcommand{\ol}{\overline}
\newcommand{\sm}{\setminus}

\newcommand{\Complex}{\mathbb{C}}

\newcommand{\0}{\emptyset}
\newcommand{\cl}{\mathrm{Cl}}
\newcommand{\Int}{\mbox{int}}
\newcommand{\bd}{\mathrm{Bd}}
\newcommand{\dia}{\mbox{diam}}
\newcommand{\sh}{\mbox{Sh}}
\newcommand{\ch}{\mathrm{Ch}}
\newcommand{\al}{\alpha}
\newcommand{\ph}{\varphi}
\newcommand{\be}{\beta}
\newcommand{\ga}{\gamma}
\newcommand{\Ga}{\Gamma}
\newcommand{\si}{\sigma}
\newcommand{\ta}{\theta}

\newcommand{\G}{\Gamma}
\newcommand{\nin}{\not\in}
\newcommand{\acc}{\mbox{Pr}}
\newcommand{\imp}{\mbox{Imp}}

\newcommand{\tl}{\mbox{TH}}

\newcommand{\hp}{\hat{p}}

\newcommand{\wc}{\widetilde{C}}
\newcommand{\wk}{\widetilde{K}}
\newcommand{\C}{\mathbb{C}}
\newcommand{\hc}{\mbox{$\mathbb{\widehat{C}}$}}
\newcommand{\D}{\mathbb{D}}

\newcommand{\bbd}{\mbox{$\mathbb{D}$}}
\newcommand{\disk}{\mathbb{D}}

\newcommand{\ucirc}{\mathbb{S}^1}
\newcommand{\uc}{\mathbb{S}^1}

\newcommand{\lam}{\mathcal{L}}

\newcommand{\Z}{\mathcal{Z}}
\newcommand{\A}{\mathcal{A}}

\newcommand{\Hc}{\mathcal{H}}
\newcommand{\Kc}{\mathcal{K}}

\newcommand{\Ta}{\Theta}
\newcommand{\iy}{\infty}
\newcommand{\g}{\mathbf{g}}

\newcommand{\h}{\mathbf{h}}

\begin{document}

\begin{frontmatter}


  \title{Locally connected models for {J}ulia sets} \author{Alexander
    M. Blokh\corauthref{cor}\thanksref{bthanks}},
  \thanks[bthanks]{This author was partially supported by NSF grant
    DMS-0456748.} \corauth[cor]{Corresponding author.}
  \ead{ablokh@math.uab.edu} \author{Clinton
    P. Curry\thanksref{cthanks}}, \thanks[cthanks]{This author was
    partially supported by NSF grant DMS-0353825.}
  \ead{clintonc@uab.edu} \author{Lex
    G. Oversteegen\thanksref{othanks}} \thanks[othanks]{This author
    was partially supported by NSF grant DMS-0405774.}
  \ead{overstee@math.uab.edu} \address{University of Alabama at
    Birmingham, Department of Mathematics, Birmingham, AL 35294-1170,
    USA}

  \begin{abstract}
    Let $P$ be a polynomial with a connected Julia set $J$. We use
    continuum theory to show that it admits a \emph{finest monotone
      map $\ph$ onto a locally connected continuum $J_{\sim_P}$},
    i.e. a monotone map $\ph:J\to J_{\sim_P}$ such that for any other
    monotone map $\psi:J\to J'$ there exists a monotone map $h$ with
    $\psi=h\circ \ph$. Then we extend $\ph$ onto the complex plane
    $\C$ (keeping the same notation) and show that $\ph$ monotonically
    semiconjugates $P|_{\C}$ to a \emph{topological polynomial
      $g:\C\to \C$}. If $P$ does not have Siegel or Cremer periodic
    points this gives an alternative proof of Kiwi's fundamental
    results on locally connected models of dynamics on the Julia sets,
    but the results hold for all polynomials with connected Julia
    sets. We also give a characterization and a useful sufficient condition
    for the map $\ph$ not to collapse all of $J$ into a point.
  \end{abstract}

\end{frontmatter}
\section{Introduction}\label{intro}

A major idea in the theory of dynamical systems is that of modeling an
arbitrary system by one which can be better understood and treated with the
help of existing tools and methods. To an extent, the entire field of symbolic
dynamics is so important for the rest of dynamical systems because symbolic
dynamical systems serve as an almost  universal model. A different example,
coming from one-dimensional dynamics, is due to Milnor and Thurston who showed
in \cite{mt} that any piecewise-monotone interval map $f$ of positive entropy
can be modeled by a piecewise-monotone interval map of constant slope $h$
(i.e., $f$ is \emph{monotonically semiconjugate} to $h$). For us however the
most interesting case is that of modeling  complex polynomial dynamical systems
on their connected Julia sets by so-called \emph{topological polynomials} on
their (topological) locally connected Julia sets. Let us now describe more
precisely what we mean.

Consider a polynomial map $P:\C\to \C$; denote by $J_P$ the \emph{Julia set} of
$P$, by $K_P$ its \emph{filled-in Julia set}, and by $U_\iy(P)=\C\sm K_P$ its
\emph{basin of attraction of infinity}. In this paper we \emph{always assume
that $J_P$ is connected}. A very-well known fact from complex dynamics (see,
e.g., Theorem 9.5 from \cite{Milnor:2006fr}) shows that there exists a
conformal isomorphism $\Psi$ from the complement of the closure of the open
unit disk $\bbd$ onto $U_\iy(P)$ which conjugates $z^d|_{{\C\sm \ol{\D}}}$ and
$P|_{U_\iy(P)}$. The $\Psi$-image $R_\al$ of the radial line of angle $\al$ in
$\C\sm \ol{\D}$ is called an \emph{(external) ray}. By \cite{douahubb85}
external rays with rational arguments \emph{land} at repelling (parabolic)
periodic points or their preimages (i.e., the rays compactify onto such
points). If $J_P$ is locally connected, $\Psi$ extends to a continuous function
$\ol{\Psi}$ which semiconjugates $z^d|_{{\C\sm \D}}$ and $P|_{\ol{U_\iy(P)}}$.

External rays have been extensively used in complex dynamics since
the appearance of the papers by Douady and Hubbard
\cite{douahubb85}. The fundamental idea of using the system of
external rays in order to construct special combinatorial structures
in the disk (called \emph{laminations} or \emph{geometric
laminations}) is due to Thurston \cite{thur85} (see also the paper
\cite{dou93} by Douady). Laminations allow one to relate the
dynamics of $P$ and the dynamics of the map $z^d|_{\uc}$. Below we
describe a few approaches to laminations.

Set $\psi=\ol{\Psi}|_{\ucirc}$ and define an equivalence relation $\sim_P$ on
$\ucirc$ by $x \sim_P y$ if and only if $\psi(x)=\psi(y)$. The equivalence
$\sim_P$ is called the \emph{($d$-invariant) lamination (generated by $P$)}.
The quotient space $\ucirc/\sim_P=J_{\sim_P}$ is homeomorphic to $J_P$ and the
map $f_{\sim_P}:J_{\sim_P}\to J_{\sim_P}$ induced by $z^d|_{\uc}\equiv \si$ is
topologically conjugate to $P|_{J_{P}}$. The set $J_{\sim_P}$ is a topological
(combinatorial) model of $J_P$ and is often called the \emph{topological Julia
set}.  On the other hand, the induced map $f_{\sim_P}:J_\sim\to J_\sim$ serves
as a model for $P|_{J_{P}}$ and is often called a \emph{topological
polynomial}. Moreover, one can extend the conjugacy between $P|_{J_{P}}$ and
$f_{\sim_P}:J_{\sim_P}\to J_{\sim_P}$ (as the identity outside $J_P$) to the
conjugacy on the entire plane. In fact, equivalences $\sim$ similar to $\sim_P$
can be defined abstractly, in the absence of any polynomial. Then they are
called \emph{($d$-invariant) laminations} and still give rise to similarly
constructed \emph{topological Julia sets $J_\sim$} and \emph{topological
polynomials $f_\sim$}.

In his fundamental paper \cite{kiwi97} Kiwi extended this to \emph{all} polynomials $P$
with no irrational neutral periodic points (called \emph{CS-points}),
including polynomials with disconnected Julia sets. In the
case of a polynomial $P$ with connected Julia set he constructed a
$d$-invariant lamination $\sim$ on $\ucirc$ such that $P|_{J_{P}}$ is
semiconjugate to the induced map $f_\sim:J_{\sim}\to J_{\sim}$ by a monotone
map $m:J_P\to J_{\sim}$ (\emph{monotone} means a map with connected point
preimages). Kiwi also proved that for all periodic points $p\in J_P$ the set
$J_P$ is locally connected at $p$ and $m^{-1}\circ m(p)=\{p\}$.

However the results of \cite{kiwi97} do not apply if a polynomial admits a
CS-point. As an example consider the following. A \emph{Cremer fixed point} is
a neutral non-linearizable fixed point $p\in J$. A polynomial $P$ is said to be
\emph{basic uniCremer} if it has a Cremer fixed point and no repelling/parabolic periodic
point of $P$ is bi-accessible (a point is called \emph{bi-accessible} if at
least two rays land it). In this case the only monotone map of $J_P$ onto a
locally connected continuum is a collapse of $J_P$ to a point
\cite{bo06,bo06b,bo08a}, strongly contrasting with \cite{kiwi97}.

The aim of this paper is to suggest a different (compared to \cite{kiwi97})
approach to the problem of locally connected dynamical models for connected
polynomial Julia sets $J_P$. Our approach works for \emph{any} polynomial $P$,
regardless of whether $P$ has CS-points or not, and is based upon continuum
theory. Accordingly, Section~\ref{model} does not deal with dynamics at all. To
state its main result we need the following definitions. Let $A$ be a
continuum. Then an onto map $\ph:A\to Y_{\ph, A}$ is said to be a \emph{finest
(monotone) map (onto a locally connected continuum)} if for any other monotone
map $\psi:A\to L$ onto a locally connected continuum $L$ there exists a
monotone map $h:Y_{\ph, A}\to L$ such that $\psi=h\circ \ph$. Observe, that in
this situation the map $h$ is automatically monotone because for $x\in L$ we
have $h^{-1}(x)=\ph(\psi^{-1}(x))$.

In general, it is not clear if a finest map exists. Yet if it does, it gives a
finest locally connected model of $A$ up to a homeomorphism. Suppose that
$\ph:A\to B$, $\ph':A\to B'$ are two finest maps. Then it follows from the
definition that a map associating points $\ph(x)\in B$ and $\ph'(x) \in B'$
with $x$ running over the entire $A$ is a homeomorphism between $B$ and $B'$.
Hence all sets $Y_{\ph, A}$ are homeomorphic and all finest maps $\ph$ are the
same up to a homeomorphism. Thus from now on we may talk of \emph{\textbf{the}
finest model $Y_A=Y$ of $A$} and \emph{\textbf{the} finest map $\ph_A=\ph$ of
$A$ onto $Y$}. In what follows we \emph{always} use the just introduced
notation for the finest map and the finest model. Call a planar continuum
$Q\subset \C$ \emph{unshielded} if it coincides with the boundary of the
component of $\C\sm Q$ containing infinity. The following is the main result of
Subsection~\ref{unshield} of Section~\ref{model}.

\begin{thm}\label{main1}
  Let $Q$ be an unshielded continuum. Then there exist the finest map
  $\ph$ and the finest model $Y$ of $Q$. Moreover, $\ph$ can be extended
  to a map $\hc\to \hc$ which maps $\iy$ to $\iy$, in $\hc\sm Q$ collapses only
  those complementary domains to $Q$ whose boundaries are collapsed by $\ph$,
  and is a homeomorphism elsewhere in $\hc\sm Q$.
\end{thm}

It may happen that the finest model is a point (e.g., this is so if the
continuum is \emph{indecomposable}, i.e. cannot be represented as the union of
two non-trivial subcontinua). In Subsection~\ref{well-crit} of
Section~\ref{model} we establish a useful sufficient condition for this not to
be the case. In Section~\ref{finmodpol} we apply Theorem~\ref{main1} to a
polynomial $P$ with connected Julia set and prove the following theorem.

\begin{thm}\label{main2}
  Let $P$ be a complex polynomial with connected Julia set $J_P$. Then the
  finest map $\ph_{J_P}=\ph$ can be extended to a monotone map $\hat \ph:
  \hc\to \hc$ so that $\hat \ph|_{\hc\sm J_P}$ is one-to-one in $U_\iy(P)$ and
  in all Fatou domains whose boundaries are not collapsed to points by $\ph$
  and $\hat \ph$
  semiconjugates $P$ and a \emph{topological polynomial} $g:\hc\to \hc$. There
  is a \emph{finest lamination} $\sim_P$ such that $g|_{\ph(J_P)}$ is conjugate
  to $f_{\sim_P}|_{J_{\sim_P}}$.
\end{thm}

In particular, $\ph_{J_P}$ semiconjugates the dynamics on $J_P$, so we have the
following diagram which commutes. (Here $\Phi$ is the quotient map corresponding
to the lamination $\sim_P$.)

\[\label{eq:commdiag}
\dgARROWLENGTH=5em
\dgARROWPARTS=8
\begin{diagram}
\node{J_P} \arrow{e,t}{P|_{J_P}} \arrow{se,t}{\varphi}
    \node{J_P} \arrow{se,b,1}{\varphi}
    \node{\mathbb S^1} \arrow{e,t}{\sigma_d} \arrow{sw,b,1}{\Phi}
    \node{\mathbb S^1} \arrow{sw,t}{\Phi}\\
\node[2]{J_\sim} \arrow{e,b}{g|_{J_\sim}} \node{J_\sim}
\end{diagram}
\]

Finally, in Section~\ref{criter} we suggest a criterion for the fact that the
finest model is non-degenerate. Given a set of angles $A\subset \uc$ denote by
$\imp(A)$ the union of impressions of angles in $A$. Also, call a set
\emph{wandering} if all its images under a specified map are pairwise disjoint.
Finally, call an attracting or parabolic Fatou domain of a polynomial
\emph{parattracting}. Essentially, the criterion is that the finest model is
non-degenerate if and only if one of the following properties holds:

\begin{enumerate}

\item there are infinitely many bi-accessible $P$-periodic points;

\item $J_P$ has a parattracting Fatou domain;

\item $P$ admits a \emph{Siegel configuration} defined later in
Definition~\ref{def-siegel} --- basically, it means that there are several
collections of angles $A_1, \dots, A_m$ such that for all $i$ the eventual
$\si_d$-image of $A_i$ is a point and the sets $\imp(A_i)$ are wandering
continua such that on the closures of their orbits the map is monotonically
semiconjugate to an irrational rotation of the circle.

\end{enumerate}

If $P$ does not have Siegel or Cremer periodic points we deduce from
our results an independent alternative proof of Kiwi's results
\cite{kiwi97}. We also obtain a few corollaries; to state them we
need the following terminology. For notions which are not defined
here see Subsection~\ref{unshield}. By a \emph{(pre)periodic point}
we mean a point with finite orbit and by a \emph{preperiodic} point
we mean a non-periodic point with finite orbit (similarly we define
preperiodic and (pre)periodic sets as well as \emph{(pre)critical}
and \emph{precritical} points). A set $A$ is \emph{(pre)critical} if
there exists $n$ such that $P^n|_A$ is not one-to-one and
\emph{non-(pre)critical} otherwise. Call $K$ a \emph{ray-continuum}
if for some collection of angles, $K$ is contained in the union of
impressions of their external rays and contains the union of
principal sets of their external rays; the cardinality of the set of
rays whose principal sets are contained in $K$ is said to be the
\emph{valence} of $K$.

We show that if there is a wandering non-(pre)critical ray-continuum $K\subset
J_P$ of valence greater than $1$ then there are infinitely many repelling
bi-accessible periodic points and the finest model is non-degenerate. In
particular, these conclusions hold if there exists a non-(pre)periodic
non-(pre)critical bi-accessible point of $J_P$. We also rely upon the finest
model to study for what (pre)periodic points $x$ we can guarantee that the
Julia set $J_P$ is locally connected at $x$; to this end we apply a recent
result \cite{bfmot10} about the degeneracy of certain invariant continua.

\noindent \textbf{Acknowledgments.} We would like to thank the
referee for useful remarks and comments.

\section{Circle laminations}\label{lamprel}

Consider an equivalence relation $\sim$ on the unit circle $\ucirc$. Classes of
equivalence of $\sim$ will be called \emph{($\sim$-)classes} and will be
denoted by boldface letters. A $\sim$-class consisting of two points is called
a \emph{leaf}; a class consisting of at least three points is called a
\emph{gap} (this is more restrictive than Thurston's definition in
\cite{thur85}; for the moment we follow \cite{bl02} in our presentation).  Fix
an integer $d>1$. Then $\sim$ is said to be a \emph{($d$-)invariant lamination}
if:

\noindent (E1) $\sim$ is \emph{closed}: the graph of $\sim$ is a closed set in
$\ucirc \times \ucirc$;

\noindent (E2) $\sim$ defines a \emph{lamination}, i.e., it is \emph{unlinked}:
if $\g_1$ and $\g_2$ are distinct $\sim$-classes, then their convex hulls
$\ch(\g_1), \ch(\g_2)$ in the unit disk $\bbd$ are disjoint,

\noindent (D1) $\sim$ is \emph{forward invariant}: for a class $\g$, the set
$\si_d(\g)$ is a class too

\noindent which implies that

\noindent (D2) $\sim$ is \emph{backward invariant}: for a class $\g$, its
preimage $\si_d^{-1}(\g)=\{x\in \ucirc: \si_d(x)\in \g\}$ is a union of
classes;

\noindent (D3) for any gap $\g$, the map $\si_d|_{\g}: \g\to \si_d(\g)$ is a
\emph{covering map with positive orientation}, i.e., for every connected
component $(s, t)$ of $\ucirc\setminus \g$ the arc $(\si_d(s), \si_d(t))$ is a
connected component of $\ucirc\setminus \si_d(\g)$.

The lamination in which all points of $\uc$ are equivalent is said to be
\emph{degenerate}. It is easy to see that if a forward invariant lamination
$\sim$ has a class with non-empty interior then $\sim$ is degenerate. Hence
equivalence classes of any non-degenerate forward invariant lamination are
totally disconnected.

Call a class $\g$ \emph{critical} if $\si_d|_{\g}: \g\to \si(\g)$ is not
one-to-one, \emph{(pre)critical} if $\si_d^j(\g)$ is critical for some $j\ge
0$, and \emph{(pre)periodic} if $\si^i_d(\g)=\si^j_d(\g)$ for some $0\le i<j$.
Let $p: \ucirc\to J_\sim=\ucirc/\sim$ be the quotient map of
$\ucirc$ onto its quotient space $J_\sim$, let $f_\sim:J_\sim \to J_\sim$ be
the  map induced by $\sigma_d$. We call $J_\sim$ a \emph{topological Julia set}
and the induced map $f_\sim$ a \emph{topological polynomial}. The set $J_\sim$
can be canonically embedded in $\C$ and then the map $p$\, can be extended to
the map $\hp: \C\to\C$ \cite{dou93}. Radial lines from $\uc$ are then mapped by
$\hp$ onto \emph{topological external rays} of the topological Julia set
$J_\sim$ on which the map $z\mapsto z^d$ induces a well-defined extension of
$f_\sim$ onto the union of $J_\sim$ and the component of $\C\sm J_\sim$
containing infinity.

We need the following theorem \cite{kiwi02}. Given a closed set
$G'\subset \uc$ let the ``polygon'' $G=\ch(G')\subset\ol{\disk}$  be
its convex hull, i.e., the smallest convex set in the disk
containing $G'$. In this case we say that $G'$ is the \emph{basis}
of $G$. In this situation let us call $G$ (and $G'$) a
\emph{wandering polygon} if the\, sets $G=\ch(G'), \ch(\si(G')),
\ch(\si^2(G')), \dots$ are all pairwise disjoint (and so the sets
$G', \si(G'), \dots$ are pairwise unlinked, see (E2) above). In
particular, if a gap $\g$ is a wandering polygon then $\g$ is not
(pre)periodic and we will call it a \emph{wandering gap}.  Also,
call $G$ (and $G'$) \emph{non-(pre)critical} if the cardinality
$|\si^n(G')|$ of $\si^n(G')$ equals the cardinality $|G'|$ of $G'$
for all $n$, and \emph{(pre)critical} otherwise.

\begin{thm}\label{kiwi-wan} If \, $G$ is a wandering polygon then
$|G'|\le 2^d$, and if\, $G$ is not (pre)critical then $|G'|\le d$.
\end{thm}

Consider a simple closed curve $S\subset J_\sim$. Call the bounded component
$U(S)=U$ of $\C\sm J_\sim$ enclosed by $S$ a \emph{Fatou domain}. By
Theorem~\ref{kiwi-wan} $S$ is (pre)periodic and for some minimal $k$ the set
$f^k_\sim(S)=Q$ is periodic of some minimal period $m$ in the sense that
pairwise intersections among sets $Q, \ldots, f_\sim^{m-1}(Q)$ are at most
finite. We cannot completely exclude such intersections; e.g., in the case of a
parabolic fixed point $a$ in a polynomial locally connected Julia set, there
will be several Fatou domains ``revolving'' around $a$ and containing $a$ in
their boundaries. However, it is easy to see that $U(Q), \dots,
U(f^{m-1}_\sim(Q))$ are pairwise disjoint.

\begin{lem}[\cite{bl02}, Lemma 2.4]\label{two-dyns}
There are only two types of dynamics of $f^m_\sim|_S$.

\begin{enumerate}

\item The map $f^m_\sim|_S$ can be conjugate to an
appropriate irrational rotation.

\item The map $f^m_\sim|_S$ can be conjugate to
$z^k|_{\uc}$ with the appropriate $k>1$.

\end{enumerate}

\end{lem}

In the case (1) we call $U$ a \emph{(periodic) Siegel domain} and in the case
(2) we call $U$ a \emph{(periodic) parattracting domain}.

The map $f_\sim$, which above was extended onto the unbounded component of
$\C\sm J_\sim$, can actually be extended onto the entire $J_\sim$-plane as a
branched covering map. Indeed, it is enough to extend $f_\sim$ appropriately
onto any bounded component $V$ of $\C\sm J_\sim$. This can be done by noticing
the degree $k$ of $f_\sim|_{\bd(V)}$ and extending $f_\sim$ onto $V$ as a
branched covering map of degree $k$ so that the extension of $f_\sim$ remains a
branched covering map of degree $d$ and behaves, from the standpoint of
topological dynamics, just like a complex polynomial. In particular, if $S$ is
a Siegel domain of period $m$, we may assume that $U(S)$ is foliated by Jordan
curves on which $f^m_\sim$ acts as the rotation by the same rotation number as
that of $f^m_\sim$. On the other hand, if $k>1$ then $f^m_\sim|_{U(S)}$ should
have one attracting (in the topological sense) fixed point to which all points
inside $U(S)$ are attracted under $f^m_\sim$. Any such extension of $f_\sim$
onto $\C$ will still be called a \emph{topological polynomial} and denoted
$f_\sim$. In Section~\ref{finmodpol} we relate $P$ and the appropriate
extension of $f_\sim$ much more precisely, however here it suffices to
guarantee the listed properties.

\begin{thm}\cite{bl02}\label{nowanco} The map $f_\sim|_{J_\sim}$ has no
wandering continua.
\end{thm}

The collection of chords in the boundaries of the convex hulls of
all equivalence classes of $\sim$ in $\disk$ is called a
\emph{($d$-invariant) geometric lamination (of the unit disk)}.
Denote the geometric lamination obtained from the lamination $\sim$
by $\lam_\sim$. In fact, geometric laminations - in what follows
\emph{geo-laminations} - can also be defined abstractly (as was
originally done by Thurston \cite{thur85}). A \emph{geometric
prelamination} $\lam$ is a collection of chords in the unit disk
called \emph{(geometric) leaves} and such that any two leaves meet
in at most a common endpoint. If in addition the \emph{union
$|\lam|$} of all the leaves of $\lam$ is closed, $\lam$ is said to
be a \emph{geometric lamination}. The closure of a component of
$\disk \sm |\lam|$ is called a \emph{(geometric) gap}. A \emph{cell
of a geometric prelamination $\lam$} is either a gap of $\lam$ or a
leaf of $\lam$ which is not on the boundary of any gap of $\lam$. If
it is clear that we talk about a geo-lamination we will use
\emph{leaves} and \emph{gaps}. Gaps of a lamination understood as an
equivalence class of an equivalence relation are normally denoted by
a small boldface letter (such as $\g$) while geometric gaps of
geometric laminations are normally denoted by capital letters (such
as $G$).

Denote a leaf $\ell=ab\in\lam$ by its two endpoints. Given a geometric gap
(leaf) $G$, set $G'=G\cap \uc$ and call $G'$ the \emph{basis of $G$}. Clearly
the boundary of each geometric gap is a simple closed curve $S$ consisting of
leaves of $\lam$ and points of $\uc$. As in \cite{thur85} one can define the
linear extension $\si^*$ of $\si$ over the leaves of $\lam$ which can then be
extended over the entire unit disk (using, e.g., the barycenters) so that not
only is $\si^*(ab)=\si(ab)$ the chord (or point) in $\ol{\disk}$ with endpoints
$\si(a)$ and $\si(b)$ but also for any geometric gap $G$ we have  that
$\si^*(G)$ is the convex hull of the set $\si(G')$. Even though we denote this
extension of $\si$ by $\si^*$, sometimes (if it does not cause ambiguity) we
use the notation $\si$ for $\si^*$ (e.g., when we apply $\si^*$ to leaves).

A geometric prelamination $\lam$ is \emph{$d$-invariant} if

\begin{enumerate}

\item  (forward leaf invariance) for each $\ell=ab\in\lam$, either
$\si(\ell)\in\lam$ or $\si(a)=\si(b)$,

\item  (backward leaf invariance) for each leaf $\ell\in\lam$ there
exist $d$ {\bf disjoint} leaves $\ell_1,\dots,\ell_d\in\lam$ such that for each
$i$, $\si(\ell_i)=\ell$,

\item   (gap invariance) for each gap $G$ of $\lam$, if $G'=G\cap \uc$
is the basis of $G$ and $H$ is the convex hull of $\si(G')$ then either $H\in
\uc$ is a point, or $H\in\lam$ is a leaf, or $H$ is also a gap of $\lam$.
Moreover, in the last case $\si^*|_{\bd(G)}:\bd(G)\to \bd(H)$ is a positively
oriented composition of a monotone map $m:\bd(G)\to S$, where $S$ is a simple
closed curve, and a covering map $g:S\to \bd(H)$.

\end{enumerate}

Clearly, $\lam_\sim$ is a geometric lamination and $\sim$-gaps are bases of
geometric gaps of $\lam_\sim$. In general, the situation with leaves and
geometric leaves is more complicated (e.g., the basis of a geometric leaf on
the boundary of a finite gap of $\lam_\sim$ is not a $\sim$-leaf). Thus in what
follows speaking of leaves we will make careful distinction between the two
cases (that of a geometric leaf and that of a leaf as a class of a lamination).
Note that Theorem~\ref{kiwi-wan}
applies to wandering (geometric) gaps of (geometric) laminations.

Slightly abusing the language, we sometimes use for gaps terminology
applicable to their bases. Thus, speaking of a \emph{finite/infinite} gap
$G$ we actually mean that $G'$ is finite/infinite. Now we study infinite
gaps (of geometric laminations) and establish some of their properties.
We begin with a series of useful general lemmas in which we
establish some properties of geometric laminations. Given two points
$x, y\in \uc$, set $\rho(x, y)$ to be the length of the smallest arc in $\uc$,
containing $x$ and $y$. There exists $\eps_d>0$ such that
$\rho(\si_d(x),\si_d(y)) > \rho(x,y)$ whenever $0<\rho(x,y) < \eps_d$.

\begin{lem}\label{lem:expand}
  If $K \subset \ucirc$ and $k > 0$ are such that $\lim_{i \rightarrow
    \infty} \dia(\si_d^{ik}(K))=0$, then there exists $i_0$ such that
  $\dia(\si_d^{i_0k}(K))=0$.
\end{lem}

\begin{proof}
  If $\lim_{i \rightarrow \infty} \dia(\si_d^{ik}(K)) = 0$, there exists $i_0$
  such that $\dia(\si_d^{ik}(K)) < \eps_{kd}$ for all $i \ge i_0$.  If $\dia(\si_d^{i_0k}(K)) \neq 0$
  then $( \dia(\si_d^{ik}(K)) )_{i=i_0}^\infty$ is an increasing sequence of
  positive numbers converging to $0$, a contradiction. So $\dia(\si_d^{i_0k}(K))=0$.
\end{proof}

Let us study geometric leaves on the boundary of a periodic gap.

\begin{lem}\label{lem:periodic_gap}
  Suppose that $G$ is a (pre)periodic gap of a geometric lamination. Then every leaf
  in $\bd(G)$ is either (pre)periodic from a finite collection of grand orbits of periodic
  leaves, or (pre)critical from a finite collection of grand orbits of critical leaves.
\end{lem}

\begin{proof}
  We may assume that the gap $G$ is fixed. Let $\ell$ be a leaf which is not
  (pre)periodic.  Since $\bd(G)$ is a simple closed curve and $\si^i(\ell) \cap
  \si^j(\ell)$ may consist of at most a point, $\lim_{i \rightarrow \infty}
  \dia(\si^i(\ell))=0$.  Therefore, by Lemma~\ref{lem:expand}, there exists
  $i_0$ such that $\dia(\si^{i_0}(\ell))=0$, meaning that $\ell$ is
  (pre)critical. Now, there are only finitely many leaves $\al\be$ in $\bd(G)$
  such that $\rho(\al, \be)\ge \eps_d$, and there are only finitely many
  critical leaves in any geometric lamination. Since by the properties of $\eps_d$ any
  non-degenerate leaf in $\bd(G)$ maps to one of them, the proof of the lemma
  is complete.
\end{proof}

In what follows a geometric leaf of a geometric lamination is called
\emph{isolated} if it is the intersection of two distinct gaps of
the lamination. It is called \emph{isolated from one side} if it is
a boundary leaf of exactly one gap of the lamination. A leaf is said
to be a \emph{limit leaf} if it is not an isolated leaf. Let us
study critical leaves of geometric laminations. The following
terminology is quite useful: a leaf is said to be \emph{separate} if
it is disjoint from all other leaves and gaps. Observe that if
$\ell$ is a separate leaf then $\ell$ is a limit leaf from both
sides. Also, if a gap or a separate leaf is such that its image is a
point we call it \emph{all-critical}. Clearly, a gap is all-critical
if and only if all its boundary leaves are critical. It may happen
that two all-critical gaps are adjacent (have a common leaf).
Moreover, there may exist several all-critical gaps whose union
coincides with their convex hull. In other words, their union looks
like a ``big'' all-critical gap inside which some leaves are added.
Then we call this union an \emph{all-critical union of gaps}.
Clearly we can talk about boundary leaves of all-critical unions of
gaps. Moreover, each all-critical gap is a part of an all-critical
union of gaps, and there are only finitely many all-critical gaps.

\begin{lem}\label{no-crit-leaf} Suppose that $\lam$ is a $d$-invariant
geo-lamination and $\ell$ is one of its critical leaves. Then one of the following
holds:

\begin{enumerate}

\item $\ell$ is isolated in $\lam$;

\item $\ell$ is a separate leaf;

\item $\ell$ is a boundary leaf of a union of all-critical gaps all boundary
leaves of which are limit leaves.
\end{enumerate}

In particular, if $\lam$ is the closure of a $d$-invariant
prelamination $\lam'$ and $\ell$ lies on the
boundary of a geometric gap $G$ of $\lam$ then either $\ell\in \lam'$, or $\si(G)$
is a point.

\end{lem}

\begin{proof} Suppose that neither (1) nor (2) holds. Then $\ell\in \lam$ is a
critical leaf lying on the boundary of a gap $G$ which is the limit of a
sequence of leaves $\ell_i$ approaching $\ell$ from outside of $G$. If $\si(G)$
is not a point, then $\si(\ell_i)$ must cross $\si(G)$, a contradiction. Hence
$\si(G)$ is a point and all leaves in the boundary of $G$ are critical. Take the all-critical union
of gaps $H$ containing $G$. If all other boundary leaves of $H$ are limit leaves we are done.
Otherwise there must exist a boundary leaf $\ell$ of $H$ and a gap $T$ to whose boundary
$\ell$ belongs. Then the leaves $\si(\ell_i)$ will cross the image $\si(T)$, a contradiction.
This completes the proof.

\end{proof}

The next lemma gives a useful condition for an infinite gap to have nice
properties. By two \emph{concatenated} leaves we mean two leaves with a common
endpoint, and by a \emph{chain of concatenated leaves} we mean a (two-sided)
sequence of leaves such that any consecutive leaves in the chain are
concatenated (such chains might be both finite and infinite). For brevity we
often speak of just \emph{chains} instead of ``chains of concatenated leaves''.

\begin{lem}\label{good-gap}
Let $G$ be an infinite gap and on its boundary there are no leaves $\ell$ such
that for some $n, m$ we have that $\si^m(\ell)$ is a leaf while
$\si^{m+n}(\ell)$ is an endpoint of $\si^m(\ell)$. Then the following claims
hold.

\begin{enumerate}

\item There exists a number $N$ such that any
chain of concatenated leaves in $\bd(G)$ consists of no more than $N$ leaves.

\item All non-isolated points of $G'$ form a Cantor set $G'_c$, and so
for any arc $[a, b]\subset \uc$ such that $[a, b]\cap G'$ is not contained in one chain,
the set $G'\cap [a, b]$ is uncountable (in particular, the basis $G'$ of $G$ is uncountable).

\item If $G$ is $\si^n$-periodic then $\si^n|_{\bd(G)}$ is semiconjugate to $\si_k:\uc\to\uc$
with the appropriate $k>0$ by the conjugacy which collapses to points all arcs in $\bd(G)$
complementary to $G'_c$. If $k=1$ the map to which $\si^n|_{\bd(G)}$ is semiconjugate is an
irrational rotation of the circle.
\end{enumerate}

\end{lem}

\begin{proof}
By Theorem~\ref{kiwi-wan}, $G$ is (pre)periodic. Since there are only finitely
many gaps in the grand orbit of $G$ on which the map $\si$ is not one-to-one,
we see that it is enough to prove the lemma with the assumption that $G$ is
fixed. Moreover, by Lemma~\ref{lem:periodic_gap} we may assume that all
periodic leaves in $\bd(G)$ are fixed with fixed endpoints. Consider a chain of
concatenated leaves from $\bd(G)$. By Lemma~\ref{lem:periodic_gap} under some
power of $\si$ this chain maps onto one of finitely many chains containing a
critical or a fixed leaf. Thus, it remains to prove the lemma for chains
containing a critical and/or a fixed leaf. By way of contradiction we may
assume that $L$ is a maximal infinite chain of concatenated leaves (it may be
one-sided or two-sided).

First let $\ell\in L$ be a fixed leaf with fixed endpoints. By the assumptions
of the lemma and by the properties of laminations each leaf concatenated to
$\ell$ also has fixed endpoints. Repeating this argument we see that the chain
consists of fixed leaves with fixed endpoints, hence $L$ is a finite chain of
fixed leaves with fixed endpoints. Second, consider the case when $\ell\in L$
is a critical leaf. Consider the points $a, b\in \uc$ with $[a, b]\subset \uc$
the smallest arc whose convex hull contains $L$. Then by Theorem~\ref{kiwi-wan}
the convex hull $\ch(L)$ of $L$ cannot be a wandering polygon. It follows that
for some $m$ we have that $\si^m(L)\subset \si^{m+n}(L)$. Since by the above
there are no leaves with periodic endpoints in $L$ and by the assumptions of
the lemma no leaf of $L$ can map into its endpoint, we see that all leaves of
$\si^m(L)$ map under $\si^n$ in the same direction, say, towards the point $a$
so that every leaf has an infinite orbit converging to $a$. However then $a$ is
$\si^n$-fixed and must repel close points under $\si^n$, a contradiction. Since
there exist only finitely many distinct chains containing a critical or
periodic leaf, there exists a number $N$ such that any chain of concatenated
leaves in $\bd(G)$ consists of no more than $N$ leaves. This immediately
implies that any non-isolated point of $G'$ is a limit point of other
non-isolated points. Hence the set $G'_c$ of all non-isolated points of $G'$ is
a Cantor set, and the claims (1) and (2) of the lemma are proven.

To prove (3) define $m:\bd(G)\to \uc$ by collapsing to points all complementary
arcs to $G'_c$ in $\bd(G)$. It follows that $(\si^*)^n|_{\bd(G)}$ is
monotonically semiconjugate by the map $m$ to a covering map $f$ of the circle
of a positive degree. It follows that for any non-degenerate arc $I\subset \uc$
the set $m^{-1}(I)\cap \uc$ is uncountable. Let us show that $I$ is not
wandering, i.e. the intervals $\{ f^k(I) \, \mid \, k > 0\}$ are not pairwise
disjoint. Indeed, if $I$ wanders under $f$ then so does $m^{-1}(I)$ under
$\sigma_d^*$. Since $\bd(G)$ is homeomorphic to $\uc$, then $\lim_{k
\rightarrow \infty}{\dia( (\sigma_d^*)^k(m^{-1}(I)))} = 0$, contradicting
Lemma~\ref{lem:expand}.

  Also, let us show that $I$ is not periodic. Suppose that $f^q(I)\subset I$.
  Then $f^m|_I$ is monotone preserving orientation and all points of $I$
  converge to an $f^q$-fixed point under $(\si^*)^q$. On the other hand, only
  countably many points of an uncountable set $m^{-1}(I) \cap \uc$ map into a
  $\si$-periodic point. Thus, there exists a \emph{non-(pre)periodic} point
  $y\in m^{-1}(I)\cap \uc$ such that $m(y)$ converges under $(\si^*)^q$ to an
  $f^q$-fixed point $z$. Since $m$ is monotone this implies that the orbit of
  $y$ approaches the interval $m^{-1}(z)$ but does not map into it (because $y$
  is non-(pre)periodic). Thus, $y$ must converge to an endpoint of $m^{-1}(z)$,
  which is impossible (e.g., it contradicts Lemma~\ref{lem:expand}). A standard
  argument now implies that $f$ is an irrational rotation or a map $\si_k$ with
  appropriately chosen $k$, still we sketch it for the sake of completeness.
  Consider two cases.

  \smallskip

  \noindent\textbf{Case 1:} $\si^*|_{\bd(G)}$ is monotone.

  \smallskip

  Let us show that $f$ has no periodic points. By way of contradiction, suppose
  $f^q(x)=x$, choose a point $y\neq x$ with $f^q(y)\neq y$, and let $I$ be the
  component of $\uc\sm \{x, y\}$ containing $f^q(y)$. Since $\si^*|_{\bd(G)}$
  is monotone, it follows that $I$ is a periodic interval, a contradiction.
  Therefore, $f:\ucirc \rightarrow \ucirc$ is a positively oriented map with no
  periodic points and no wandering intervals, and is therefore conjugate to an
  irrational rotation by \cite[Theorem 1.1]{Melo:1993nx}. By
  Lemma~\ref{lem:periodic_gap} all leaves in $\bd(G)$ are (pre)critical.

   \smallskip

  \noindent \textbf{Case 2:} $\si^*|_{\bd(G)}$ is not monotone.

   \smallskip

  Since $f$ is a covering map of degree $k>1$ without periodic and wandering
  intervals, $f$ is conjugate to $z \mapsto z^k$ for some $k$. Indeed, that
  there is a monotone semiconjugacy between $f$ and $\si_k$ is well-known (see,
  e.g., \cite{mr07} for the case $k=2$). However if there are no wandering
  intervals and periodic intervals, then the semiconjugacy cannot collapse any
  intervals and is therefore a conjugacy. In what follows the semiconjugacy
  which we have just defined in both cases will be denoted $\psi$.

\end{proof}

Given a geo-lamination $\lam$, a periodic geometric gap $G$ satisfying
conditions of Lemma~\ref{good-gap} is called a \emph{Fatou gap (domain) of
$\lam$}. If $G$ is a Fatou domain, then by Theorem~\ref{kiwi-wan} $G'$ is
(pre)periodic. A Fatou domain $G$ is called \emph{periodic (preperiodic,
(pre)periodic))} if so is $G'$. A periodic Fatou domain $G$ of period $m$ is
called \emph{parattracting} if $(\si^*)^m|_{\bd(G)}$ is not monotone (in the
topological sense introduced earlier in the paper) and \emph{Siegel} otherwise.
Equivalently, $G$ is parattracting (resp. Siegel) if $(\si^*)^m|_{\bd(G)}$ can
be represented as the composition of a covering map of degree greater than $1$
(resp. equal to $1$) and a monotone map. The \emph{degree of $\si^m|_G$} is
then defined as the degree of the model map $f$ defined in
Lemma~\ref{good-gap}. Thus, the terms ``parattracting Fatou domain'' and
``Siegel domain'' are used both for the geometric laminations and for the
topological polynomials. Since it will always be clear from the context which
situation is considered, this will not cause any ambiguity in what follows.

There are several cases in which Lemma~\ref{good-gap} applies. The first one is
considered in Lemma~\ref{fat-gap-domain}. Recall, that given a lamination
$\sim$ we denote by $p$ the corresponding quotient map $p:\uc\to J_\sim$.
Recall also, that for a lamination understood as an equivalence relation we denote
its gaps (i.e., classes with more than two elements) by small bold letters
(such as $\g$).

\begin{lem}\label{fat-gap-domain}
Suppose that $\g$ is an infinite gap of a non-degenerate lamination $\sim$.
Then $B=\bd(\ch(\g))$ contains no geometric (pre)critical leaves and therefore
is a Fatou gap. In addition to that, any chain of concatenated geometric leaves
in $B$ eventually homeomorphically maps to a periodic chain, and if $\g$ is
periodic of period $n$ then the degree of $(\si^*)^n|_B$ is greater than $1$.
\end{lem}

\begin{proof} By Theorem~\ref{kiwi-wan} $\g$ is (pre)periodic.
Suppose that $\ell=\al\be\subset B$ is a critical geometric leaf and that
$\g\subset [\al, \be]$. By Lemma~\ref{no-crit-leaf} $\ell$ cannot be a limit
leaf of $\lam_\sim$. Hence there is a geometric gap $H$ of $\lam_\sim$ on the
side of $\ell$ opposite to $B$ (so that $H'\subset [\be, \al]$). The points
$\al, \be$ are limit points of $H'$ for otherwise there must exist a geometric
leaf $\be\ga$ or $\ta\al$ and hence $\ga$ must be added to $\g$, a
contradiction. By the gap invariance then $\si(H)=\si(\ch(\g))$. Now, since $H$
is a gap of $\lam_\sim$, either $H'$ is a class itself, or there are
uncountably many distinct $\sim$-classes among points of $H'$. However the
latter is impossible because all these classes map into one $\sim$-class $\g$.
Thus, $H'$ is one $\sim$-class which implies that it had to be united with $\g$
in the first place, a contradiction. Hence Lemma~\ref{good-gap} applies to
$\g$. Clearly, it follows also that any chain of concatenated geometric leaves
in $B$ eventually homeomorphically maps to a periodic chain.

Let us now prove the last claim of the lemma. Since there are no critical
leaves in $B$, $(\si^*)^n|_B$ is a covering map. If the degree of
$(\si^*)^n|_B$ is $1$, then $\si^n|_\g$ is one-to-one. By a well-known result
from the topological dynamics (see, e.g., Lemma~18.8 from \cite{Milnor:2006fr})
$\g$ must be finite, a contradiction.
\end{proof}

Lemma~\ref{fat-gap-domain} shows that if $\sim$ is a lamination, then there are
two types of Fatou domains of $\lam_{\sim}$: 1) Fatou domains whose basis (the
intersection of the boundary with $\uc$) is one $\sim$-class (one
$\sim$-gap), in which case the Fatou domain corresponds to a cutpoint in the quotient space; or 2) Fatou domains for which this is not true (and which
correspond to a Fatou domain in the $J_\sim$-plane). However this distinction
cannot always be made if we just look at the geometric lamination.

For a lamination $\sim$ the induced geo-lamination $\lam_\sim$ has the property
that every geometric leaf is either disjoint from all other geometric leaves
and gaps, or contained in the boundary of a unique geometric gap $G$. For an
arbitrary geometric lamination, this is no longer the case. Hence, in general
distinct geometric gaps may intersect. If, given a geo-lamination $\lam$,
$\sim$ is a lamination such that $a\sim b$ whenever $ab=\ell\in \lam$, we say
that the lamination $\sim$ \emph{respects the geo-lamination $\lam$}. Given a
$d$-invariant geo-lamination $\lam$, let $\approx=\approx_\lam$ be the
\emph{finest lamination which respects $\lam$}. It is not difficult to see that
$\approx$ is unique and $d$-invariant. Let $\pi:\uc\to J_\approx$ be the
corresponding quotient map. It may well be the case that $\uc/\approx$ is a
single point (see \cite{bcmo08} for a characterization of quadratic geometric
laminations $\lam$ with non-degenerate $J_{\approx_\lam}$).

Let us discuss the properties of $\approx$. It is shown in \cite{bcmo08} that
$\approx$ can be defined as follows: $a\approx b$ if and only if there exists a
continuum $K\subset \uc\cup |\lam|$ containing $a$ and $b$ such that $K\cap
\uc$ is countable. By Lemma~\ref{good-gap} if $G$ is a Fatou domain of $\lam$,
then $G/\approx$ is a simple closed curve. In particular, whenever a
$d$-invariant geo-lamination $\lam$ contains a Fatou gap, then $J/\approx_\lam$
is non-degenerate. Moreover, if $F$ is an invariant Fatou domain, then the
restricted map $f_\approx:\pi(\bd(F))\to\pi(\bd(F))$ coincides with the map $f$
from Lemma~\ref{good-gap} and is conjugate to a either an
irrational rotation of a circle (if $F$ is Siegel) or
to the map $\si_m$ for $m$ equal to the degree of $\si|_F$ (in the parattracting case).
The case of a periodic Fatou domain is similar.

Suppose that $\mathcal A$ is a forward invariant family of pairwise
disjoint periodic or non-(pre)critical wandering gaps/leaves with a given family of
their preimages so that together they form a collection
$\G_{\mathcal A}$ of sets (basically, this is a collection of sets
from the grand orbits of elements of $\mathcal A$). The leaves from
the boundaries of sets of $\G_{\mathcal A}$ form a $d$-invariant
geometric prelamination $\lam_{\mathcal A}$. Clearly, the sets from
the collection $\G_{\mathcal A}$ are cells of $\lam_{\mathcal A}$.
The prelamination $\lam_{\mathcal A}$ and its closure
$\ol{\lam_{\mathcal A}}$ (which is a geo-lamination \cite{thur85})
are said to be \emph{generated} by $\mathcal A$ (then $\mathcal A$
is called a \emph{generating family}). The following important
natural case of this situation was studied by Kiwi in \cite{kiwi97}.

Given a point $y\in J_P$, denote by $A(y)$ the set of all angles whose rays
land at $y$. If $J_P$ is locally connected then $A(y)\ne \0$ for any $y\in
J_P$, however otherwise this is not necessarily so.  A point $y\in J_P$ is called
\emph{bi-accessible} if $|A(y)|>1$ (i.e., there are at least two rays landing
at $y$). By Douady and Hubbard \cite{douahubb85} if $x$ is a repelling or
parabolic periodic point (or a preimage of such point) then $A(x)$ is always
non-empty, finite, and consists of rational angles. Denote by $R$ the set of
all its periodic repelling (parabolic) \emph{bi-accessible} points and their
preimages. Let $x\in R$; also, given a set $A$ denote by $\ch(A)$ its convex
hull. Then let $G_x=\ch(A(x))$ and let $|\lam_{rat}|$ be the union of all the
sets $G_x, x\in R$. Let $\lam_{rat}$ be the collection of all chords contained
in the boundaries of all the sets $G_x$. Then $\lam_{rat}$, called the
\emph{rational geometric prelamination}, is a $d$-invariant geometric
prelamination. By \cite{thur85} the closure $\ol{\lam_{rat}}$ of $\lam_{rat}$
in the unit disk is a closed $d$-invariant geo-lamination called \emph{rational
geometric lamination}.

The situation described above may be considered in a more general way. Suppose
that we are given a geometric prelamination generated by (pre)periodic or
wandering non-(pre)critical pairwise disjoint gaps and leaves. Then any result
concerning its closure will serve as a tool for studying $\ol{\lam_{rat}}$. The
following theorem can be such a tool. If we require that all gaps or leaves in
such prelamination map onto their images in a covering fashion, we can conclude
that there are no critical leaves in the prelamination. Indeed, such leaves can
only belong to gaps/leaves disjoint from other leaves and collapsing to a point
(\emph{all-critical}). However we assume that the generating family consists of
leaves and gaps which are non-(pre)critical. Hence an all-critical cell of
the prelamination cannot come from the forward orbits of the elements of the
generating family. On the other hand, it cannot come from their backward orbits
since the generating family consists of gaps and leaves (and the image of an
all-critical gap/leaf is a point).



In Lemma~\ref{atmost4} we deal with geometric laminations. For simplicity, in
its proof speaking of leaves and gaps we actually mean geometric leaves and
gaps. By a \emph{separate leaf} we mean a leaf disjoint from all other leaves
or gaps.

\begin{lem}\label{atmost4}
Suppose that $\lam^-$ is a non-empty geometric $d$-invariant pre-lamination
generated by a generating family $\mathcal A$ such that no cell of $\lam^-$
contains a critical leaf on its boundary. Let $\lam$ be the closure of the
prelamination $\lam^-$. Then the following holds.

\begin{enumerate}

\item\label{if3}
If three leaves of $\lam$ meet at a common endpoint, then the leaf in
the middle is either a leaf from $\lam^-$ or a boundary leaf of a gap from
$\lam^-$.

\item\label{4meet}
At most four leaves of $\lam$ meet at a common end point, and if they do then
the two in the middle  are on the boundary of a gap of $\lam^-$.

\item \label{intersect}
Suppose  $G$ is a gap of $\lam$ and $xa$ is a leaf of $\lam$ such
that $xa\cap G=\{x\}$. Let $xb\subset G$ be the leaf such that $xb$
separates $G\sm xb$ from $xa\sm \{x\}$.  Then either $xb$ is a leaf
from $\lam^-$, or $G$ is a cell of $\lam^-$, or there exists a gap
$H$ of $\lam^-$ such that $xa\cup xb \subset H$. In particular, if
two gaps $G,H$ of $\lam$ meet only in a point, then there exists a
gap $K$ of $\lam^-$ such that both $G\cap K$ and $H\cap K$ are
leaves from $\lam$.

\item\label{allcrit}
Suppose that $\ell$ is a critical leaf of $\lam$. Then either $\ell$
is a separate leaf and all its images are disjoint from $\ell$, or
$\ell$ is a boundary leaf of an all-critical gap $H$ of $\lam$, all
boundary leaves of $H$ are limit leaves, and $\si^n(H)$ is disjoint
from $H$ for all $n>0$. In particular, $\ell$ is a limit leaf from
at least one side.

\item \label{perio-strict} If $G$ is a gap or leaf of $\lam$ and
$\si^n(G)\subset G$ then $\si^n(G)=G$.

\item \label{wanper} Any gap or leaf $G$ of $\lam$ either wanders or is such
that for some $m<n$ we have $\si^m(G)=\si^n(G)$.

\item \label{uncount}
If $G$ is a gap of $\lam$ such that $G'$ is infinite, then $G$ is a Fatou gap.
\end{enumerate}

\end{lem}

\begin{proof}
(\ref{if3}) Suppose that $ax,bx,cx$ are three leaves of $\lam$ with $x<a<b<c$
in the counterclockwise order $<$ on $S^1$. Then $bx$ is isolated. Hence  $bx$
is either a separate leaf from $\lam^-$ or a boundary leaf of a gap from
$\lam^-$.

(\ref{4meet}) Suppose that $\lam$ contains the  leaves $a_1x,\ a_2x,\
a_3x,\dots, a_n x$ with $n\ge 4$. We may assume that $a_1<a_2\dots<a_n<x$.
Since the leaves $a_2x,a_3x,\dots, a_{n-1}x$ are isolated and must come from
$\lam^-$. Since cells of  $\lam^-$ are pairwise disjoint, $n=4$ and the
leaves $a_2x,a_3x$ are on the boundary of a gap of $\lam^-$.

(\ref{intersect}) Suppose that a gap $G$ and a leaf $xa$ of $\lam$
meet only at $x$. Let $xb\subset G$ be the leaf which separates
$G\sm xb$ from $xa\sm\{x\}$. Then $xb$ is isolated and, hence,
either a separate leaf in $\lam^-$ or a boundary leaf of a gap of
$\lam^-$. If the former holds, or if $G$ is a gap of $\lam^-$, we
are done. Otherwise there exists a gap $H$ of $\lam^-$ which
contains $xb$. It now follows easily that $xa\subset H$ as desired.

(\ref{allcrit}) The first part immediately follows from
Lemma~\ref{no-crit-leaf} and the assumption that there are no critical leaves
in the prelamination $\lam^-$. This implies that the point $\si(H)$ is
separated by leaves of $\lam^-$ from all other points of $\uc$. Hence by the
properties of geo-laminations $\si^n(H)\subset H$ is impossible.

(\ref{perio-strict}) By (\ref{allcrit}) we may assume that $G$ is a
gap which contains no $\si^n$-critical leaves in $\bd(G)$ and
$\si^n(G)$ is not a point. Now, if $G$ is a gap and $\si^n(G)=ab$ is
a boundary leaf of $G$ then $\si^{2n}(a)=a, \si^{2n}(b)=b$ and $G$
is a finite gap. Denote by $ca$ the other leaf in $\bd(G)$
containing $a$. Suppose first that $G$ is a cell of $\lam^-$. Then
$\si^n(G)=ab$ is a leaf of $\lam^-$ strictly contained in the
boundary of $G$, a contradiction. Hence some boundary leaves of $G$
may come from $\lam^-$, but there are no two consecutive leaves like
that in $\partial G$. Thus, if $ca$ is a leaf of $\lam^-$ then
$\si^n(ca)=ab$ (because there are no critical leaves in $\bd(G)$) is
also a leaf of $\lam^-$ and we get a contradiction by the above.
Thus, $ca$ is not a leaf of $\lam^-$ which implies that $ca$ is not
on the boundary of a gap $H\ne G$ (otherwise $ca$ is isolated in
$\lam$ and hence $ca$ must be a leaf of $\lam^-$, a contradiction).
We conclude that $ca$ is a limit leaf from the outside of $G$.
However then the $\si^{2n}$-images of leaves converging to $ca$ will
cross $G$, a contradiction.

(\ref{wanper}) Suppose that $G$ is a gap or leaf from $\lam$ for
which the conclusions of the lemma do not hold. If $G$ is infinite,
then by Theorem~\ref{kiwi-wan} $G$ is preperiodic. So we may assume
that $G$ is finite and that $|G'|=|\si^i(G)'|$ for all $i>0$. By the
assumption about $G$ we may assume that for some $n>0$, $G\cap
\si^n(G)\ne\0$ and $\si^n(G)\ne G$ (in particular, $G$ is not a
point because otherwise we would have $\si^n(G)=G$) and for no $i\ne
j$ we have $\si^i(G)=\si^j(G)$. Since the cells of $\lam^-$ are
pairwise disjoint, $G$ is not a cell of $\lam^-$. Moreover, it is
easy to see that no leaf in $\bd(G)$ is periodic. Indeed, otherwise
under the map, which fixes the endpoints of this leaf, $G$ will have
to be mapped onto itself (see ~\ref{perio-strict}).

Suppose now that $G$ is a leaf. If an endpoint of $G$ is $\si^n$-fixed then
we would have more than 4 leaves of $\lam$ coming out of this point,
contradicting (\ref{4meet}). Hence $\si^n(G)$ must be a leaf, ``concatenated''
to $G$, $\si^{2n}(G)$ is a leaf ``concatenated'' to $\si^n(G)$, and so on.
Since these leaves do not intersect inside $\disk$, it follows that they
converge to a leaf or to a point $\lim_i \si^{ni}(G)$ which is $\si^n$-fixed.
This contradicts the fact that $\si^n$ is locally repelling.

Suppose next that $G$ is a gap such that $G$ and $\si^n(G)$ meet
along the (isolated) leaf $\ell$. By ~\ref{perio-strict}, $\si^n(G)$
is a gap. Hence the leaf $\ell$ is isolated in $\lam$ which implies
that $\ell\in\lam^-$. Since there are no periodic leaves in
$\bd(G)$, $\si^n(\ell)\cap\ell=\0$. Repeating this argument we see
that leaves $\si^{in}(\ell)$ are such that the gaps $\si^{in}(G)$
are ``concatenated'' (attached) to each other at these leaves. This
again, as in the previous paragraph, implies that the limit $\lim
\si^{ni}(\ell)$ exists and is either a $\si^n$-fixed leaf in $\lam$
or a $\si^n$-fixed point in $\uc$. This contradicts the fact that
$\si$ is locally repelling.

Hence it remains to consider the case when $G$ and $\si^n(G)$ meet in a point
$x\in \uc$. By (\ref{4meet}) there exist boundary leaves $xa\subset\bd(G)$ and
$xb\subset\bd(\si^n(G))$ and there exists a gap $H\in\lam^-$ which contains
both of these leaves in its boundary. If $H$ is periodic, then $xa$ and $xb$
are periodic too, a contradiction. Hence, $H$ is not periodic. Since
$H\in\lam^-$, $H$ must wander and $\si^n(H)\cap \si^m(H)=\0$ when $n\ne m$. It
follows that sets $\si^{in}(G)$ are all ``concatenated'' at points $x,
\si^n(x), \dots$, the set $\bigcup_{i=0}^\infty \si^{ni}(G)$ is connected set,
and $\lim\si^{ni}(G)$ exists and is either a leaf in $\lam$ or point in $\uc$
which is fixed under $\si^n$. As before, this contradicts the fact that $\si^n$ is
locally repelling and completes the proof of (\ref{wanper}).

(\ref{uncount}) Follows immediately from (\ref{allcrit}) and
Lemma~\ref{good-gap}.(2).

\end{proof}

We are ready to construct a non-degenerate lamination compatible with $\lam^-$
(or, equivalently, with $\lam$). Suppose that $\mathcal A=\{G_\al\}$ is a
generating collection of finite gaps/leaves and $\lam^-$ is a non-empty
geometric $d$-invariant pre-lamination generated by a generating family
$\mathcal A$ such that there are no critical leaves in $\lam^-$. Set
$\lam=\ol{\lam^-}$ and $\approx_{\lam}=\approx_{\mathcal A}$.

\begin{thm}\label{periodic-prel}
We have that $\uc/\approx_{\mathcal A}$ is non-degenerate and any equivalence
class of $\approx_{\mathcal A}$ is finite. Moreover, $\approx_\A$ has no Siegel
domains. In particular, if $R\ne \0$, then the finest lamination
$\approx_{rat}$ which respects $\lam_{rat}$, is not degenerate and in the
geometric lamination $\lam_{\approx_{rat}}$ every leaf not contained in the
boundary of a Fatou domain is a limit of leaves from $\lam_{rat}$.
\end{thm}






\begin{proof}
Let $\approx$ be the equivalence relation in $\uc$ defined as follows:
$x\approx y$ if and only if there exists a continuum $K\subset S^1\cup |\lam|$
such that $x,y\in K$ and $K\cap S^1$ is countable (such continua are called
\emph{$\omega$-continua}). Then $\approx$ is the finest closed equivalence
relation which respects $\lam$; moreover, $\approx$ is an invariant lamination
(\cite{bcmo08}).

Now, suppose first that $\lam$ has no gaps. Then the leaves from $\lam$ fill
the entire disk. If there are two leaves coming out of one point, then there
must be infinitely many leaves coming out of the same point which is impossible
by Lemma~\ref{atmost4}. Hence all leaves of $\lam$  are pairwise disjoint and
equivalence classes of $\approx$ are endpoints of (possibly degenerate) leaves.
From now in the proof we assume that $\lam$ has gaps. It now follows easily
that gaps of $\lam$ are dense in $\disk$ and so if an $\omega$-continuum $K$
meets a leaf $\ell\in \lam$ and $\disk\sm \ell$, then $K$ must contain one of
the endpoints of $\ell$.

In the proof below we construct so-called \emph{super gaps} and associate them
to some leaves and gaps of $\lam$. If $G$ is a leaf of $\lam$ disjoint from all
gaps of $\lam$ we call it a \emph{separate} leaf (of $\lam$). In this case put
$G^+=G$ and call it a \emph{super gap} associated with $G$. Clearly, $G^+$ is a
two sided limit of leaves from $\lam^-$. Let $\mathfrak{G}=\bigcup\{G\mid G
\text{ is a \emph{finite} gap of }\lam\}$. For any gap $G$ of $\lam$, let $G^+$
be the closure of the convex hull of the component of $\mathfrak{G}$ which
contains $G$. Again, call $G^+$ a \emph{super gap} associated with $G$. By
Lemma~\ref{atmost4}.(\ref{wanper}), a gap/leaf $G$ of $\lam$ either wanders or
is such that $\si^m(G)=\si^n(G)$ for some $m<n$.

\smallskip


\noindent \textbf{Claim 1.} \emph{Suppose that $G$ is a non-(pre)critical wandering gap/leaf
of $\lam$. Then $G^+$ is either a separate leaf or a finite union of finite
gaps whose convex hull is a non-(pre)critical wandering polygon and every
leaf  in its boundary is a limit of leaves from $\lam^-$.}

\smallskip

\noindent \emph{Proof of Claim} 1.\, The case when $G$ is a separate leaf
immediately follows from the definition of a super gap; in this case $G^+=G$ is
a separate leaf. Suppose next that $G$ is a leaf which meets a gap $H$ of
$\lam$ or $G$ is a non-(pre)critical wandering gap. By Lemma~\ref{atmost4}.(\ref{4meet}), there
exist gaps $G_0, \dots, G_n$ such that $G\subset G_0$, $G_i\cap G_{i+1}$ is a
leaf and $G_n=H$ (if $G\cup H$ is a wandering gap, then we set $G_0=H=G_n$). By
Lemma~\ref{atmost4}.(\ref{wanper}) $G_0$ is either non-(pre)critical wandering or (pre)periodic,
and since $G$ is non-(pre)critical wandering, so is $G_0$.

Assume, by way of induction, that $G'$ is a finite union of finite gaps which
is a non-(pre)critical wandering polygon and $H$ is a gap of $\lam$ which meets $G'$ along the
leaf $ab$. Then $ab$ is non-(pre)critical wandering because it comes from $G'$. Again since by
Lemma~\ref{atmost4} $H$ is either non-(pre)critical wandering or (pre)periodic, we see that $H$
also wanders. In particular, by Theorem~\ref{kiwi-wan} $H$ is finite. We claim
that $H\cup G'$ is a non-(pre)critical wandering polygon. For suppose this is not the case. Then
we may assume that $\si(G')\cap H\ne\0$. Moreover, the common leaf $ab$ of $G'$
and $H$ is isolated and hence comes from $\lam^-$. Therefore it is not critical
and its image $\si(ab)$ is a leaf again. Clearly, $\si(ab)$ is the leaf shared
by $\si(G')$ and $\si(H)$. Repeating this argument, we get a sequence of gaps of
$\lam$ ``concatenated'' at images of the leaf $ab$. Similarly to the arguments
in the proof of Lemma~\ref{atmost4}.(\ref{wanper}) it implies that the orbit of
$ab$ converges to a point or a leaf but never maps into it which is impossible
because of repelling properties of $\si$.


It follows that $G^+$ is a non-(pre)critical wandering polygon and, by Theorem~\ref{kiwi-wan},
$|G^+\cap S^1|\le 2^d$. Hence $G^+$ is finite union of finite gaps. Note that
every leaf on the boundary of $G^+$ is a limit of leaves from $\lam^-$ as
desired. This completes the proof of Claim 1. \qed

Observe that by Lemma~\ref{atmost4}.(\ref{uncount}) for every infinite gap $G$
of $\lam$ the set $G'$ consists of a Cantor set $G'_c\subset \uc$ and a
countable collection of finite sets $G'_1, G'_2, \dots$ of cardinality at most
$k$ ($k$ depends on $G$) such that for every $i$ the set $G'_i$ is the
intersection of $G'$ and a complementary to $G_c$ subarc $U_i$ of $\uc$. If
$|G'_i|>1$ we connect the endpoints of $U_i$ with a leaf $\ell$ and add $\ell$
to the lamination $\lam$. It is easy to see that the resulting extension of the
geo-lamination $\lam$ is a geo-lamination itself. From now on we will use the
notation $\lam$ for the new extended geo-lamination.

Suppose next that $G$ is a finite (pre)periodic gap or a (pre)periodic leaf of
$\lam$. If some forward image of $G$ contains a critical leaf on its boundary,
then we may assume that $\si(G)$ is a point by
Lemma~\ref{atmost4}.(\ref{allcrit}). Hence each leaf in the boundary of $G$ is
a limit of leaves from $\lam^-$ and $G^+=G$. If no forward image of $G$
contains a critical leaf on its boundary, then from some time on
$|\si^k(G')|>1$ stabilizes and by Lemma~\ref{atmost4}.(\ref{wanper}) we may
assume that $\si^m(G)=G$ for some $m>0$ and $|G'|\ge 2$. Choose $n\ge 0$ such
that $\si^n(G)=G$ and each leaf in the boundary of $G$ is fixed.

\smallskip

\noindent \textbf{Claim 2.} \emph{Suppose $G$ is a (pre)periodic finite gap or
(pre)periodic leaf of $\lam$. Then $G^+$ is a finite polygon and any leaf in
the boundary of $G^+$ is either a limit of leaves from $\lam^-$ or is contained
in an uncountable gap of $\lam$. Moreover, if $G$ is an $n$-periodic gap/leaf
then $G^+\supset G$ is the convex hull of a subset of the component of the set
of leaves from $\lam$ with $\si^n$-fixed endpoints.}

\smallskip

\noindent \emph{Proof of Claim} 2.\, Suppose that $\si^n(G)=G$ and that all
points of $G'$ are fixed. If $G$ is a separate leaf then $G^+=G$ and we are
done. If $G$ is a non-separate leaf then it is a boundary leaf of a gap $Q$.
Since the endpoints of $G$ are $\si^n$-fixed, either $\si^n(Q)=G$ or, because
the map $\si^n|_\bd(Q)$ is positively oriented, $\si^n(Q)=Q$. The former is
impossible by Lemma~\ref{atmost4}.(\ref{perio-strict}). Hence we find a gap
$Q\supset G$ whose all vertices are $\si^n$-fixed. Finally, if $G$ is a gap
then we can set $Q=G$. Thus, if $G$ is not a separate leaf, we can always find
a gap $Q\supset G$ whose all vertices are $\si^n$-fixed.

Suppose, by induction, that $G$ is a finite polygon which is a finite union of
gaps from $\lam$. Moreover, suppose that the boundaries of the gaps consist of
leaves with $\si^n$-fixed endpoints. Let $H$ be any gap of $\lam$ which meets
$G$ along the leaf $ab$. By Lemma~\ref{atmost4}.(\ref{perio-strict}) $\si^n(H)$
cannot be equal to $ab$, and since $\si^n|_{\bd(H)}$ is a positively oriented
covering map we see that $\si^n(H)=H$. If $H$ is finite, all leaves in the
boundary of $H$ must also be fixed. Otherwise $H$ is an infinite, and hence
uncountable, gap. It follows that $G^+\supset G$ is a finite union of
$\si^n$-fixed gaps and that every leaf in the boundary of $G^+$ is  either a
limit of leaves from $\lam^-$ or is contained in an uncountable gap from
$\lam$.

Now let $G$ be any (pre)periodic finite gap of $\lam$. Then there exists $n$
such that $\si^n(G)=H$ is periodic. If $H$ is a point then by
Lemma~\ref{atmost4}.(\ref{allcrit}) all leaves in the boundary of $G$ are limit
leaves and  hence $G^+=G$. Otherwise if $H$ is a separate leaf it follows that
all boundary leaves of $G$ are limit leaves and we are done. Thus by the
previous paragraph we may assume that $H^+$ is a finite union $H=H_1,\dots,H_n$
of gaps of $\lam$. Let $G^+$ be the component of $\bigcup_i \si^{-n}(H_i)$
which contains $G$. Then $G^+$ is a finite union of finite gaps and every leaf
in the boundary of $H^+$ is either a limit of leaves from $\lam^-$ or on the
boundary of an uncountable gap from $\lam$.  \qed

We now pass on to the proof of the fact that $\approx_\lam=\approx_{\mathcal
A}$ is non-degenerate. Indeed, consider all the super gaps constructed in Claim
1 and Claim 2 (i.e., all the sets $G^+\cap \uc$ for different gaps and leaves
$G$ of the geo-lamination $\lam$). Also, if a point $x\in \uc$ does not belong
to any gap or leaf of $\lam$ we call it a \emph{separate point} and add it to
the family of sets which we construct. Clearly, all sets in the just
constructed family $\mathcal F$ of super gaps and separate points are closed.
Moreover, by the definition two sets in the family $\mathcal F$ are disjoint.
Indeed, two super gaps cannot meet over a leaf by the definition. If they meet
at a vertex then by Lemma~\ref{atmost4}.(\ref{if3}),
Lemma~\ref{atmost4}.(\ref{4meet}) and by the construction of the extended
lamination $\lam$ they again must be in one super gap. Hence all elements of
$\mathcal F$ are pairwise disjoint.

Considering elements of $\mathcal F$ as equivalence classes we get a closed
equivalence $\approx$ on $\uc$ which respects $\lam^-$ and $\lam$ (it is easy
to see that $\approx$ is indeed closed). By the construction and Claims 1 and
2, all $\approx$-classes are finite. Because of the definition of a super gap,
if an equivalence respects $\lam^-$ (and hence $\lam$), it cannot split a
$\approx$-class (i.e., a set $G^+\cap \uc$ for some gap/leaf $G$ of $\lam$)
into two or more classes of equivalence. Therefore $\approx$ is the finest
equivalence which respects $\lam^-$. As was explained in the beginning of the
proof of Theorem~\ref{periodic-prel}, by \cite{bcmo08} there always exists the
finest equivalence which respects a geometric lamination, and from what we have
just proven if follows that this finest equivalence $\approx_\lam$ coincides
with $\approx$. By Claims 1 and 2 super gaps are finite, thus all
$\approx$-classes are finite and hence $\approx$ is non-degenerate.

Finally assume that $U$ is a Siegel domain of $\approx$. Then $\bd(U)$ must
contain a critical leaf because otherwise by a well-known result from the
topological dynamics (see, e.g., Lemma~18.8 from \cite{Milnor:2006fr})
$\bd(U)\cap \uc$ must be finite, a contradiction. However by
Lemma~\ref{atmost4}.(\ref{allcrit}) this is impossible. The rest of
Theorem~\ref{periodic-prel} which deals with the rational lamination follows
immediately from the construction. This completes the proof of the theorem.
\end{proof}

\section{The existence of a locally connected model for unshielded planar continua}\label{model}

As outlined in Section~\ref{intro}, in this section we prove
Theorem~\ref{main1} and show the existence of the finest model and the finest
map for any unshielded planar continuum $Q$. We do this in
Subsection~\ref{unshield}. In Subsection~\ref{well-crit} we suggest a
topological condition sufficient for an unshielded continuum $Q$ to have a
non-degenerate finest model. This will be used later when in
Theorem~\ref{main2} we establish the criterion for the connected Julia set of a
polynomial to have a non-degenerate finest model.

\subsection
{The existence of the finest map $\ph$ and the finest locally connected model}
\label{unshield}

In what follows $Q$ will always denote an unshielded continuum in the plane and
$U_\iy$ will always denote the corresponding simply connected neighborhood of
infinity in the sphere, called the \emph{basin of infinity} (so that
$Q=\bd(U_\iy)$).

We begin by constructing the finest monotone map $\ph$ of $Q$ onto a locally
connected continuum.  The map will be constructed in terms of impressions of
the continuum $Q$. Since $Q = \bd(U_\infty)$, there is a unique conformal
isomorphism $\Psi : U_\iy \rightarrow \D$ which has positive real derivative at
$\infty$. (Note that the domain of the map is in the dynamical plane.) Define the \emph{principal set of the external angle $\alpha \in
\ucirc$} as
\[\acc(\alpha) = Q \cap
\overline{ \Psi^{-1}( \{re^{2 \pi i \alpha} \, \mid \, r \in [0,1)\}
)}.\]

Define the \emph{impression
of the external angle $\alpha \in \ucirc$} as
\[ \imp(\alpha) = \Big\{ \lim_{i \rightarrow \infty}
  \Psi^{-1}(\alpha_i) \, \mid \, \{\alpha_i \, \mid \, i > 0\} \subset \D
\text{ and }\lim_{i \rightarrow \infty} \alpha_i = \alpha \Big\}. \]

The \emph{positive wing (of an impression)} is defined as follows:

\[
\begin{split}
  \imp^+(\alpha) = \Big\{ \lim_{i \rightarrow \infty}
  \Psi^{-1}(\alpha_i) \, \mid \, \{\alpha_i \, \mid \, i > 0\}
  \subset \D \text{ and }&\\
  \lim_{i \rightarrow \infty} \alpha_i = \alpha \text{ with }
  \arg(\alpha_i) \ge \arg(\alpha)& \Big\}.
\end{split}
\]

Similarly, the \emph{negative wing (of an impression)} is defined as follows:
\[
\begin{split}
  \imp^-(\alpha) = \Big\{ \lim_{i \rightarrow \infty}
  \Psi^{-1}(\alpha_i) \, \mid \, \{\alpha_i \, \mid \, i > 0\}
  \subset \D \text{ and }&\\
  \lim_{i \rightarrow \infty} \alpha_i = \alpha \text{ with }
  \arg(\alpha_i) \le \arg(\alpha)& \Big\}.
\end{split}
\]
The differences between these sets are illustrated in
Figure~\ref{fig:impressions}.
\begin{figure}
  \centering
  \includegraphics[width=0.45 \textwidth]{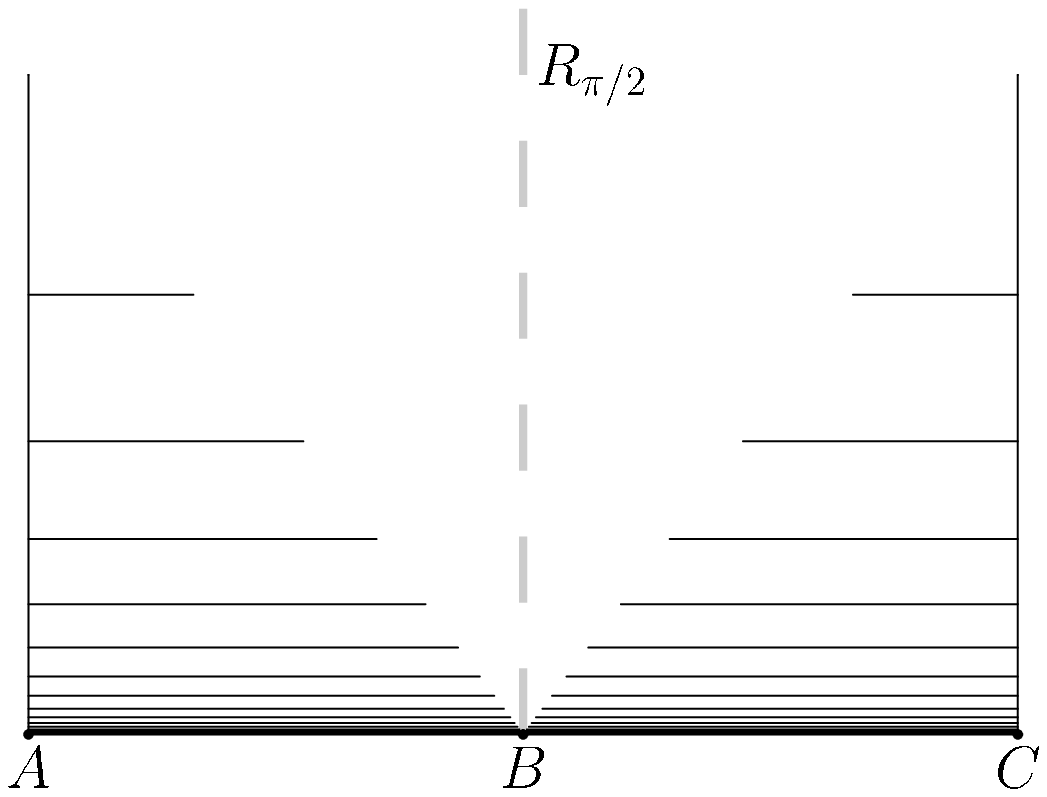}
  \hfill
  \includegraphics[width=0.45 \textwidth]{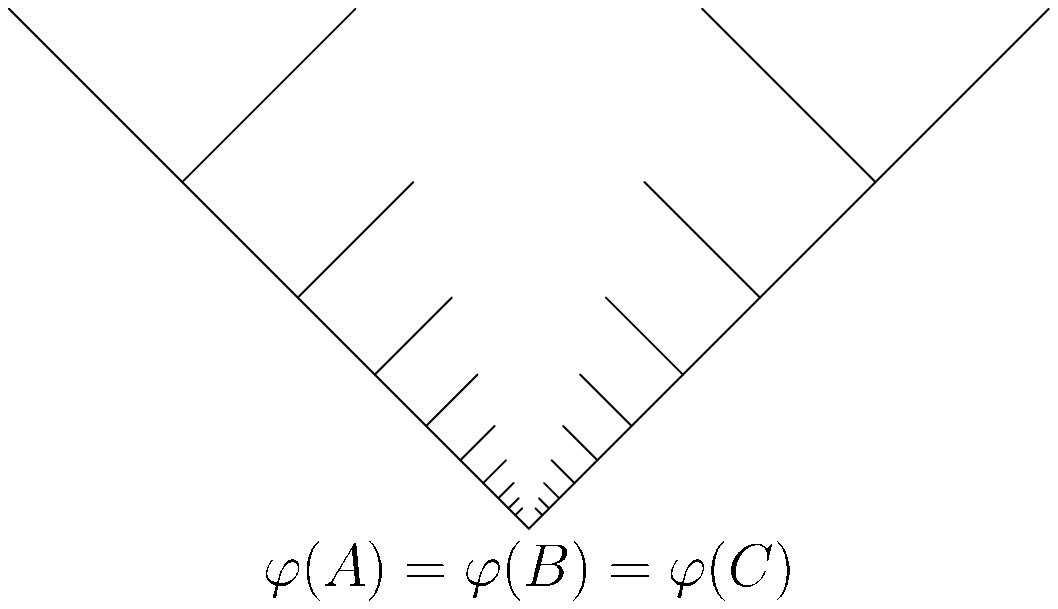}
  \caption{On the left is depicted a continuum with an external ray
    for which the impression, positive wing, negative wing, and
    principal sets are distinct.  The positive wing is the line
    segment joining $A$ and $B$, while the negative wing is the line
    segment joining $B$ and $C$.  On the right is depicted the
    quotient by $\mathcal D$ defined in
    Lemma~\ref{lem:exists_partition}, which is locally connected.}
  \label{fig:impressions}
\end{figure}
In a lot of applications it is crucial that in the above construction the map
$\Psi$ is conformal. However the construction can be carried out if instead of
$\Psi$ certain homeomorphisms $\Psi':U_\iy\to \disk$ are used. The definitions
of the principal set, impression and wings of impression can be given in this
case as well. Since some continua we construct have topological nature, we use
this idea in what follows defining for them the map $\Psi'$ in a topological
way and then defining principle sets, impressions and wings of impressions
accordingly.

Any angle's principle set, impression, wings of its impression are
each subcontinua of $Q$.  It is known that $\acc(\alpha) =
\imp^+(\alpha) \cap \imp^-(\alpha) \subset \imp^+(\alpha) \cup
\imp^-\alpha = \imp(\alpha)$. If $Q$ is locally connected, the
impression of every external angle is a point, and therefore
impressions intersect only when they coincide. Non-locally connected
continua may have impressions of different external angles which
intersect and do not coincide. Suppose that $\mathcal D$ is a
partition of a compactum $K$ (i.e., a collection of pairwise
disjoint subsets of $K$ whose union is all of $K$). Clearly, $\D$
defines an equivalence relation on $K$ whose classes are elements of
$\D$. A partition $\D$ is called \emph{upper semi-continuous} if
this equivalence relation is closed (i.e., its graph is closed in
$K\times K$).

  \begin{lem}\label{lem:exists_partition}
    There exists a partition $\mathcal D_Q=\mathcal D$ of $Q$ which is finest
    among all upper semi-continuous partitions whose elements are unions of
    impressions of $Q$.  Further, elements of $\mathcal D$ are subcontinua of
    $Q$.
  \end{lem}

  \begin{proof}
    Let $\Xi$ be the collection of closed equivalence relations on $Q$ such
    that, for any equivalence relation $\approx$ from $\Xi$, $\imp(\alpha)$ is
    contained in one class of equivalence for any external angle $\alpha$. Then
    the equivalence relation $\bigcap \Xi$ is also an element of $\Xi$ (classes
    of equivalence of $\bigcap \Xi$ are intersections of classes of equivalence
    of all equivalence relations from $\Xi$). Let $\mathcal D$ be the
    collection of equivalence classes of $\bigcap \Xi$.

    To see that the elements of $\mathcal D$ are connected, we can define a
    finer partition $\mathcal D'$ whose elements are connected components of
    elements of $\mathcal D$.  Then $\mathcal D'$ is an upper semi-continuous
    monotone decomposition of $Q$ \cite[Lemma 13.2]{nad92}.  Since impressions
    are connected subsets of $Q$, that each impression belongs to an element of
    $\mathcal D$ implies that it belongs to an element of $\mathcal D'$.
    Therefore, $\mathcal D' \in \Xi$ and $D'$ is a refinement of $\mathcal D$,
    so $\mathcal D = \mathcal D'$, and the elements of $\mathcal D$ are
    connected.
  \end{proof}

  We will show that $Q / \mathcal D$ is locally connected, and
  $\mathcal D$ is the finest upper semi-continuous partition of $X$
  into connected sets with that property. Thus, the finest monotone
  map respecting impressions turns out to be the finest monotone map
  producing a locally connected model. To implement our plan we
  study properties of monotone maps of unshielded continua. First we
  suggest the canonic extension of any monotone map of a planar
  unshielded continuum $Q$ to a monotone map of the entire plane
  onto the entire plane. Given any monotone map $\psi$, let call
  sets $\psi^{-1}(y)$ \emph{$\psi$-fibers}, or just \emph{fibers}.

  \begin{defn}
    Let $U \subset \hc$ be a simply connected open set containing $\infty$. If
    $A$ is a continuum disjoint from $U$, the \emph{topological hull} $\tl(A)$
    of $A$ is the union of $A$ with the bounded components of $\hc \setminus
    A$. Equivalently, $\tl(A)$ is the complement of the unique component of
    $\hc\sm A$ containing $U$. Note that $\tl(A) \subset \C$ is a continuum
    which does not separate the plane.
  \end{defn}

  \begin{figure}
    \centering
    \includegraphics[width=0.45 \textwidth]{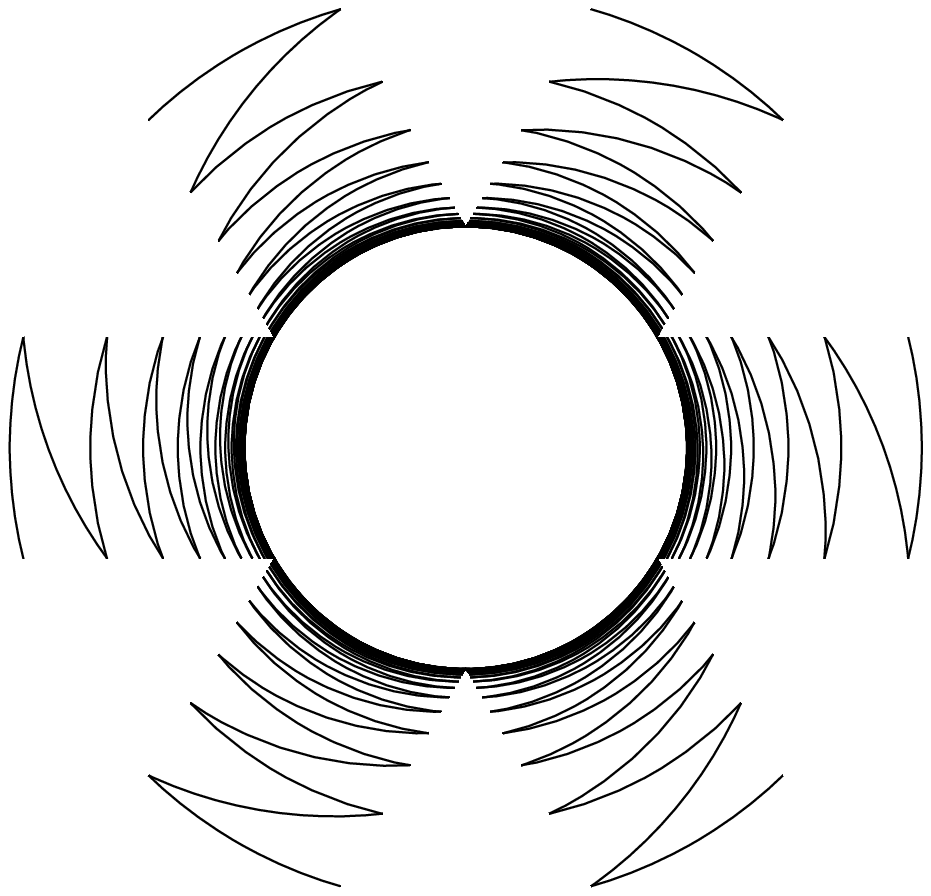}
    \includegraphics[width=0.45 \textwidth]{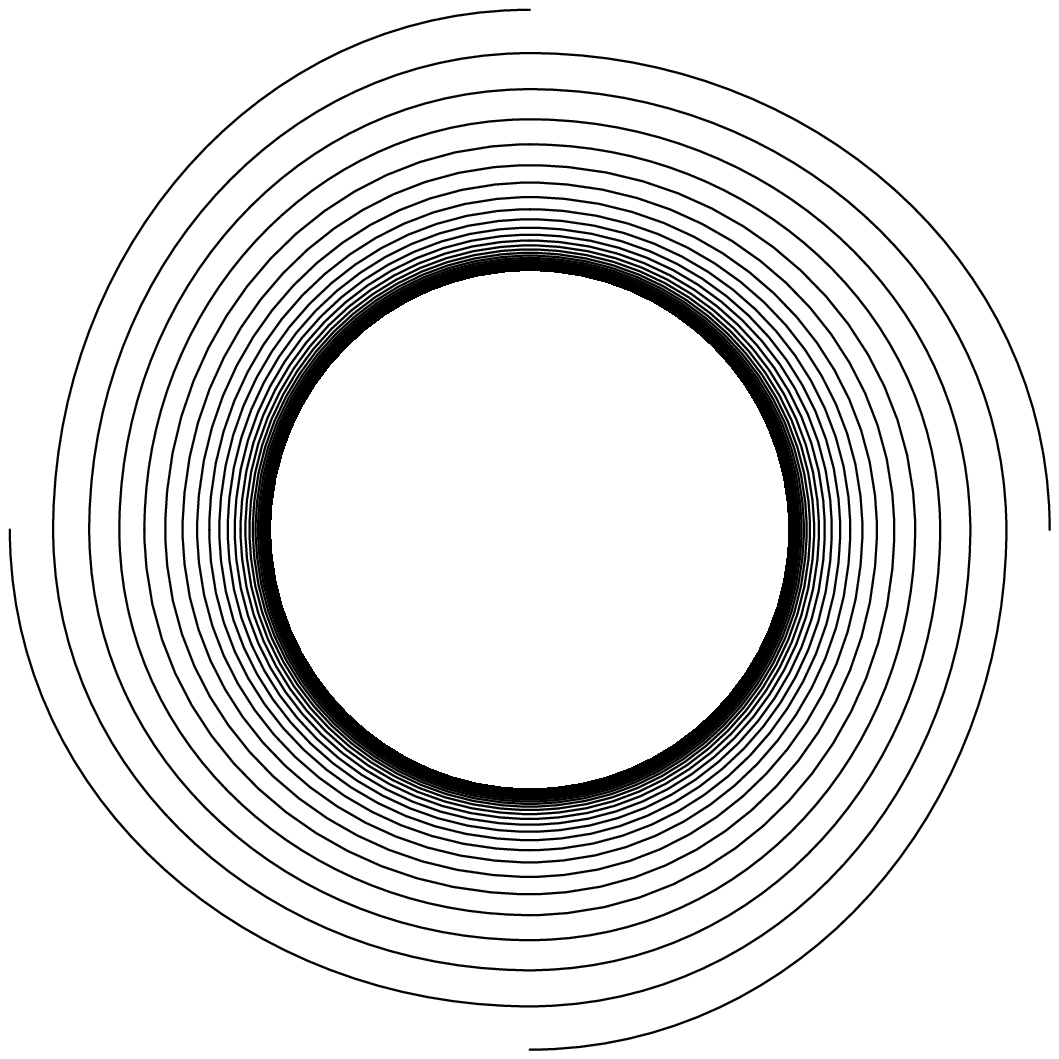}
    \caption{Here are two examples of continua for which one of the fibers
      of the finest map $\varphi$ to a locally connected continuum is a simple closed curve.  Notice that points of the
      simple closed curve in the figure on the left are accessible from
      both the bounded and unbounded complementary domains.}
    \label{fig:sepfiber}
  \end{figure}

  Suppose that a monotone map $m$ from an unshielded continuum $Q$
  to an arbitrary continuum $Y$ is given. Then $m$-fibers may be
  separating, as indicated in Figure~\ref{fig:sepfiber}, or
  non-separating. Denote by $T_m(Q)$ the union of $Q$ and the
  topological hulls of all separating fibers. To extend our map $m$
  onto the plane as a monotone map, we must collapse topological
  hulls of separating fibers because otherwise the extension will
  not be monotone. This idea is implemented in the next lemma.

  \begin{lem}\label{lem:monotone_maps_extend}
    If $Q$ is an unshielded continuum and $m:Q \rightarrow Y$ is a surjective monotone map
    onto an arbitrary continuum $Y$, then there exists a
    monotone map $M: \hc \rightarrow \hc$ and an embedding $h:Y \rightarrow \C$
    such that:

    \begin{enumerate}

    \item $M|_{\hc \setminus T_m(Q)}$ is a homeomorphism onto its
      image;

    \item $M(U_\iy)$ is a simply connected open set whose boundary is
      $M(Q)$, with $M(\infty) = \infty$; and

    \item $M|_{Q} = h \circ m$.

    \end{enumerate}

  \end{lem}

  \begin{proof}
    We extend the map $m$ by filling in its fibers. Define the collection

    \[ \widehat{\mathcal D} = \left\{\tl(m^{-1}(y)): y \in
      Y\right\} \cup \left\{\{p\}:p \nin T_m(Q)\right\}.\]

    \noindent It is immediate that $\widehat{\mathcal D}$ is an upper
    semi-continuous partition of $\hc$ whose elements are non-separating
    continua. Therefore, by \cite{m25}, $\hc / \widehat{\mathcal D}$ is
    homeomorphic to $\hc$, and there exists a monotone map $M:\hc \rightarrow
    \hc$ whose fibers are sets from $\widehat{\mathcal D}$. Observe that by the
    construction $M^{-1}(Y)=T_m(Q)$.

    Further, since points of $\hc \setminus M^{-1}(Y)$ are elements of
    $\widehat{\mathcal D}$, invariance of domain gives that $M|_{\hc \setminus
    M^{-1}(Y)}$ is a homeomorphism onto its image and that $M(U_\iy)$ is an
    open subset of $\hc$ with $M(\iy)=\iy$.  Also, $M(U_\iy) \cap M(Q) =
    \emptyset$, so $\bd(M(U_\iy)) = M(Q)$.  Finally, notice that the fibers of
    $M|_{Q}$ are the same as the $m$-fibers so there exists a natural
    homeomorphism $h:Y \rightarrow M(Q)$. This is a homeomorphism of $Y$ onto
    $M(Q)$ and an embedding of $Y$ into $\C$ since $Y$ is compact.
  \end{proof}

  Next we show that \emph{any} monotone map of an unshielded continuum
  onto a locally connected continuum must collapse impressions to
  points. A \emph{crosscut} of $Q$ is a homeomorphic image $C\subset
  U_\iy$ of an open interval $(0, 1)$ under a homeomorphism $\psi:[0,
  1]\to \C$ such that $\psi(0)\in Q\ne \psi(1)\in Q$. Define $\sh(C)$
  (the \emph{shadow of $C$}) as the closure of the bounded component
  of $U_\iy \setminus C$. Observe that in our definition of a crosscut
  and its shadow we always assume that the continuum is unshielded and
  that crosscuts are contained in the basin of infinity.

  \begin{figure}
    \centering
    \includegraphics{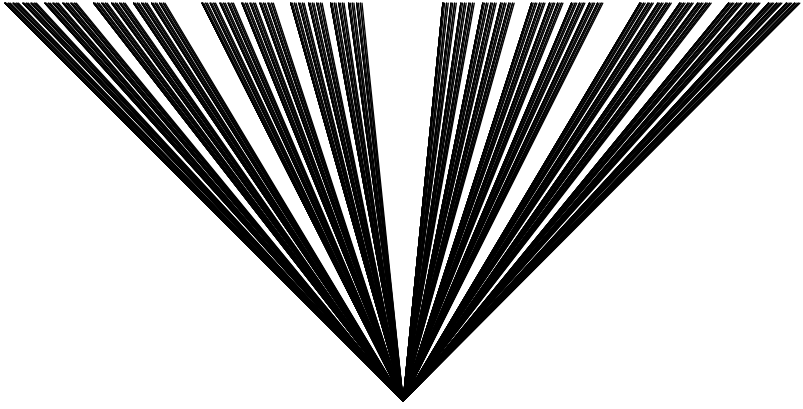}
    \caption{In this continuum, constructed by joining every point of
      a Cantor set to a base point with a straight line segment, every
      pair of non-degenerate impressions intersect, and every point is
      contained in a non-degenerate impression.  Therefore,
      Lemma~\ref{lem:impressions_to_points} concludes that the finest
      locally connected model is a point.}
    \label{fig:cantorcone}
  \end{figure}

  \begin{lem}\label{lem:impressions_to_points}
    Suppose that $m:Q \rightarrow Y$ is a monotone map onto a locally connected
    continuum. Then $m(\imp(\alpha))$ is a point for every $\alpha \in \ucirc$.
  \end{lem}

  \begin{proof}
    Let $M$ be as guaranteed in Lemma~\ref{lem:monotone_maps_extend}. Since
    $M|_{U_\iy}$ is one-to-one, it is then easy to see that a crosscut of $Q$
    maps by $M$ either to a crosscut of $M(Q)$ or to an open arc in $M(U_\iy)$
    whose closure is a simple closed curve meeting $M(Q)$ in a single point.
    Because $M(\infty) = \infty$, we see that $M(\sh(C))=\sh(M(C))$ for any
    crosscut $C$ whose image is a crosscut while if $\ol{M(C)}$ is a simple
    close curve then $M(\sh(C))$ is the interior of the Jordan disk enclosed by
    $\ol{M(C)}$.

    Choose any external angle $\alpha$.  There exists a sequence of
    crosscuts $(C_i)_{i=1}^\infty$ such that their diameters
    converge to $0$ and $\bigcap_{i=1}^\infty \sh(C_i) =
    \imp(\alpha)$ \cite[Lemma 17.9]{Milnor:2006fr}.  Since
    $\sh(C_i)$ are nested, we have

    \begin{align*}
      M(\imp(\alpha)) &=
      M\left(\bigcap_{i=1}^\infty \sh(C_i)\right)\\
      &= \bigcap_{i=1}^\infty M(\sh(C_i))\\
      &= \bigcap_{i=1}^\infty \sh(M(C_i)).
    \end{align*}

    By uniform continuity, $\lim_{i \rightarrow
      \infty}\dia(M(C_i))=0$.  Since $M(Q)$ is locally connected,
    $\bigcap_{i=1}^\infty \sh(M(C_i))$ is indeed a point, and so is
    $M(\imp(\alpha))$.
  \end{proof}

  The next lemma is essentially a converse of
Lemma~\ref{lem:impressions_to_points}.

  \begin{lem}\label{lem:collapse_impressions_lc}
    Suppose that $m:Q \rightarrow Y$ is a monotone surjective map such that
    $m(\imp(\alpha))$ is a point for all $\alpha \in \ucirc$.  Then $Y$ is
    locally connected. Moreover, the map $\Phi_m:\ucirc \rightarrow Y$ defined
    by $\Phi_m = m \circ \imp$ is a continuous single-valued onto function.
  \end{lem}

  \begin{proof}
    $\Phi_m$ is a single-valued function, since by assumption $m$ maps
    impressions to points of $Y$.  Also, it is surjective, since $m$ is
    surjective and every point is contained in the impression of some angle. To
    see sequential continuity, observe that

    \begin{align*}
      \alpha_i \rightarrow \alpha & \implies \limsup_{i \rightarrow
        \infty} \imp(\alpha_i) \subset \imp(\alpha)\\
      &\implies \limsup_{i \rightarrow \infty} m(\imp(\alpha_i))
      \subset
      m(\imp(\alpha))\\
      &\implies \Phi(\alpha_i) \rightarrow \Phi(\alpha).
    \end{align*}

    The continuous image of a locally connected continuum is locally connected,
    so $Y$ is locally connected as the $\Phi_m$-image of $\uc$.
  \end{proof}

  The picture which follows from the above lemmas is as follows. Imagine that
  we have a monotone map $m$ of an unshielded continuum $Q\subset \C$ onto a
  locally connected continuum $Y$.  By Lemma~\ref{lem:monotone_maps_extend} we
  can think of $m$ as the restriction of a monotone map $M:\hc \to \hc$ which
  in fact is a homeomorphism on $U_\iy$ as well as on the components of $\C\sm
  Q$ whose boundaries are not collapsed by $m$. To avoid confusion, we call the
  plane containing $Q$ the \emph{$Q$-plane}, and the plane containing $Y$ the
  \emph{$Y$-plane}. Likewise, if there is no ambiguity we will call various
  objects in the $Q$-plane $Q$-rays etc while calling corresponding objects in
  the $Y$-plane $Y$ rays etc.

  Now, take external conformal $Q$-rays. Then the map $M$ carries them over to
  the $Y$-plane as just continuous rays (obviously, our construction is purely
  topological and does not preserve the conformal structure in any way). The
  construction however forces all these $Y$-rays to land; moreover, the family
  of $Y$-rays can be used to define impressions in the sense of that family
  (see our explanation following the definition of the impression). By
  Lemma~\ref{lem:impressions_to_points}, these impressions must be degenerate.

  We are ready to prove the existence of the finest locally connected model and
  the finest map for unshielded continua. Recall that $\mathcal D_Q=\mathcal D$
  denotes the finest among all upper semi-continuous partitions of $Q$ whose
  elements are unions of impressions of $Q$ (it is provided by
  Lemma~\ref{lem:exists_partition}).

  \begin{thm}\label{thm:exists_finest}
    There exists a monotone map $\ph:\hc \rightarrow \hc$ such that $\ph|_Q$ is
    the finest monotone map of $Q$ onto a locally connected continuum, $\ph(Q)$
    is the finest locally connected model of $Q$, and $\ph$ is a homeomorphism
    on $\hc \setminus \ph^{-1}(\ph(Q))$. Moreover, the map $\ph|_Q$ can be
    defined as the quotient map $Q \rightarrow Q / \mathcal D$.
  \end{thm}

  \begin{proof}
    Let us show that the quotient map $m:Q \rightarrow Q / \mathcal D$ is the
    finest map of $Q$ onto a locally connected continuum. Indeed, suppose that
    $\psi:Q\to A$ is a monotone map onto a locally connected continuum $A$.
    Then $\psi$ generates an upper semi-continuous partition of $Q$ whose
    elements, by Lemma~\ref{lem:impressions_to_points}, are unions of
    impressions of $Q$. By the choice of $\mathcal D$ there exists a continuous
    map $h: Q / \mathcal D\to A$ which associates to an element $B$ of
    $\mathcal D$ the point $x\in A$ such that $\psi^{-1}(x)$ contains $B$. To
    complete the proof we let $\ph:\C \rightarrow \C$ be the extension of $m$
    guaranteed by Lemma~\ref{lem:monotone_maps_extend}.
  \end{proof}

  Define $\Phi:\ucirc \rightarrow \ph(Q)$ as $\Phi = \ph \circ \imp$. From
  Lemma~\ref{lem:collapse_impressions_lc}, $\Phi$ is a well-defined continuous
  function.
  According to the picture given after Lemma~\ref{lem:collapse_impressions_lc},
  $\Phi$ maps an angle $\al$ to the landing point of the corresponding
  $\ph(Q)$-ray (i.e., the $\ph$-image of the external conformal ray to $Q$ in
  the $Q$-plane). Then the \emph{finest lamination $\sim_Q$} (corresponding to
  $Q$) is the equivalence relation $\sim$ on $\ucirc$, defined by $\alpha_1
  \sim \alpha_2$ if and only if $\Phi(\alpha_1) = \Phi(\alpha_2)$.


\subsection{A constructive approach}\label{well-crit}

Recall that the finest map of an unshielded continuum $Q$ is always denoted by
$\ph=\ph_Q$. Fibers under the finest map will be called \emph{K-sets}. In the
notation from Subsection~\ref{unshield} and Lemma~\ref{lem:exists_partition},
K-sets are exactly the elements of the partition $\mathcal D_Q=\mathcal D$, the
finest among all upper semi-continuous partitions whose elements are unions of
impressions of $Q$. Classes of equivalence in the lamination $\sim_Q$ will be
called \emph{K-classes}. We are interested in the structure of K-sets, and will
describe how to determine if two points lie in the same K-set. Given a set
$A\subset \uc$ let $\imp(A)$ be the union of impressions of all angles in $A$.

  \begin{lem}\label{lc}
    If $\{a\}$ is a degenerate K-set then $Q$ is locally connected at $a$.
  \end{lem}

  \begin{proof}
    Suppose that $A$ is a K-class with a degenerate K-set $\imp(A)=\{a\}$ (by
    the definitions, this is equivalent to $\ph^{-1}(\ph(a))=\{a\}$). Take the
    point $\ph(a)$. Since $\ph(Q)$ is locally connected, there is a nested
    sequence of open connected neighborhoods $U_1\supset U_2\supset \dots$ of
    $\ph(a)$ such that $\cap^\iy_{i=1} U_i=\{\ph(a)\}$. By the properties of
    $\ph$, the sets $V_i=\ph^{-1}(U_i)$ form a nested sequence of open
    connected neighborhoods of $a$ with the intersection coinciding with
    $a=\ph^{-1}(\ph(a))$. So, $Q$ is locally connected at $a$.
  \end{proof}

  Now we introduce two important notions.

  \begin{defn}\label{raycont}
    A \emph{ray-compactum} (or \emph{ray-continuum}) $X \subset Q$ is a
    compactum (respectively, a continuum or a point) for which there exists a
    closed set of angles $\Theta(X) \subset \ucirc$ such that
    \[
    \bigcup_{\theta \in \Theta(X)}\acc(\theta) \subset X \subset
    \bigcup_{\theta \in \Theta(X)}\imp(\theta).
    \]
    Denote $X \cup \bigcup_{\theta \in \Theta(X)}R_\theta$
    by $\widetilde{X}$.
  \end{defn}

  One of the notions defined below is fairly standard. We give two equivalent
  definitions of the second notion, one involving separation of sets and the other
  involving cutting the plane.

  \begin{defn}
    A set $Y$ \emph{separates} a space $X$ between subsets $A$ and $B$ if $X
    \setminus Y = U \cup V$, where $A \subset U$, $B \subset V$, and $\overline
    U \cap V = U \cap \overline V = \emptyset$. We say that a ray-compactum $C$
    \emph{ray-separates} subsets $A$ and $B$ of $Q$ if $\widetilde{C}$
    separates $\overline{U_\iy}$ between $A$ and $B$. If $X\subset Q$ is a
    continuum and there are at least two points of $X$ which are ray-separated
    by $C$, we say that $C$ \emph{ray-separates} $X$.
  \end{defn}

  The definition of ray-separation can be equivalently given as follows: (a) a
  ray-compactum $C$ \emph{ray-separates} subsets $A$ and $B$ of $Q$ if $C\cap
  (A\cup B)=\0$ and there exists no component of $\ol{U_\iy}\sm \wc$
  containing points of $A$ and $B$.
  All these notions are important ingredients of the central notion of
  \emph{well-slicing}.

  \begin{defn}
    A continuum $X \subset Q$ is \emph{well-sliced} if there exists a
    collection $\mathcal C$ of pairwise disjoint ray-compacta in $Q$ such that

    \begin{enumerate}


    \item each $C \in \mathcal C$ ray-separates $X$,

    \item for every different $C_1, C_2 \in \mathcal C$ there exists
    $C_3 \in \mathcal C$ which ray-separates $C_1$ and $C_2$, and

    \item $\mathcal C$ has at least two elements.

    \end{enumerate}

    The family $\mathcal C$ is then a \emph{well-slicing family} for $X$.
  \end{defn}

  We will also use the following combinatorial (laminational) version of
  well-slicing.

  \begin{defn} Suppose that there is a collection $\mathcal C$ of at least two
  pairwise disjoint geometric leaves or gaps in $\disk$. Suppose that for every
  different $C_1, C_2 \in \mathcal C$ there exists $C_3 \in \mathcal C$ which
  separates $\disk$ between $C_1$ and $C_2$. Then the family $\mathcal C$ is
  then a \emph{well-slicing family} for $\disk$. Equivalently, consider the
  family $\mathcal C'$ of closed pairwise unlinked subsets of $\uc$. Suppose
  that for every different $C'_1, C'_2\in \mathcal C'$ there exists $C'_3\in
  \mathcal C'$ which separates $\uc$ between $C'_1$ and $C'_2$. Then we say
  that $\mathcal C'$ is a well-slicing family of $\uc$. Clearly, if
  $\mathcal C$ is a well-slicing family of $\disk$ then the intersections of
  elements of $\mathcal C$ with $\uc$ (i.e., their bases) form a well-slicing
  family of $\uc$, and vice versa.
  \end{defn}

  As an example of a well-slicing family, take $Q=\uc$. We define the family
  of subsets
  \[
   C_\alpha = \{e^{2 \pi i \alpha}, e^{-2 \pi i \alpha}\}
  \]
  with $\alpha$ taking a rational value in $[0,1/2)$.  Each $C_\alpha$ is then
  a ray-compactum with the set of angles $\Ta(C_\al)=\{\al, -\al\}$. Then for
  $0 \le \alpha < \beta < \frac 1 2$, we see that $C_\alpha$ and $C_\beta$ are
  ray-separated by $C_{(\alpha + \beta)/2}$. Hence, $\mathcal C$ is a
  well-slicing family for $\ucirc$. Set $\mathcal C_{\uc}=\mathcal C$ and call
  this collection the \emph{vertical collection}.

  Suppose that a collection $\mathcal C'$ of closed pairwise unlinked subsets
  of $\uc$ is a well-slicing family of $\uc$. Moreover, suppose that for each set
  $C'\in \mathcal C'$ the set $\imp(C')$ is a continuum in $Q$, and for
  distinct sets $C'_1, C'_2$ their impressions are disjoint. Then it follows
  from the definitions that the sets $\imp(C'), C'\in \mathcal C'$ form a
  well-slicing family of the entire $Q$. If $X\subset Q$ is such that all sets
  $A$ from this collection cut $X$ (i.e., $X\sm A$ is disconnected) then it
  follows that this is a well-slicing family for $X$.

 \begin{lem}\label{lem:middle_separates}
   Suppose that $C_1, C_2$ are disjoint ray-compacta each of which
   ray-separates $A,B \subset Q$.  If $C_3$ is a ray-compactum disjoint from $A
   \cup B$ which ray-separates $C_1$ and $C_2$, then $C_3$ also ray-separates
   $A$ and $B$.
 \end{lem}

 \begin{proof} Suppose that $C_3$ does not ray-separate $A$ and $B$.
 Then there exists a component $V$ of $\C\sm \wc_3$ containing points of both
 $A$ and $B$. Since $C_3$ ray-separates $C_1$ and $C_2$, one of these sets
 (say, $C_1$) is disjoint from $V$. Then $V$ is contained in a component $W$ of
 $\C\sm \wc_1$. Hence $W$ contains points of both $A$ and $B$ and so $\wc_1$
 does not separate $X$ between $A$ and $B$, a contradiction.
 \end{proof}

 The next lemma is close in spirit to Lemma~\ref{lem:middle_separates}.

 \begin{lem}\label{lem:sep_sep}
   Let $A, B \subset Q$.  Suppose that $K_1$ is a ray-compactum which
   ray-separates $A$ and $B$, and $K_2$ is a ray-compactum disjoint from $B$
   which ray-separates $A$ and $K_1$.  Then $K_2$ ray-separates $A$ and $B$.
  \end{lem}

  \begin{proof}
  Suppose that $K_2$ does not ray-separate $A$ and $B$. Then there exists a
  component $V$ of $\ol{U_\iy}\sm \wk_2$ containing points of both $A$ and $B$.
  Since $K_1$ ray-separates $A$ and $B$, there must be points of $K_1$ in $V$
  too. However this implies that $K_2$ does not ray-separate $A$ and $K_1$, a
  contradiction.
  \end{proof}

  The next lemma shows that elements of a well-slicing family are separated by
  infinitely many elements of the same family.

  \begin{lem}\label{lem:no_finite}
    If $\mathcal C$ is a well-slicing family of a continuum $X \subset Q$ then,
    for any two elements $C_1$ and $C_2$, infinitely many different elements of
    $\mathcal C$ separate $C_1$ and $C_2$.
  \end{lem}

  \begin{proof} Choose $C_3\in \mathcal C$ which ray-separates $C_1$ and $C_2$.
  Then choose $C_4\in \mathcal C$ which ray-separates $C_3$ and $C_2$. It is easy to see
  that $C_4\ne C_1$. By Lemma~\ref{lem:sep_sep} $C_3$ ray-separates $C_1$ and $C_2$.
  Inductively applying this argument, we will find a sequence of pairwise distinct
  elements of $\mathcal C$ each of which ray-separates $C_1$ and $C_2$ as desired.
  \end{proof}

  Now we prove the first theorem of this subsection which implies that in a few
  cases certain subcontinua of $Q$ do not collapse under the finest map $\ph$.

  \begin{thm}\label{th:wc-non-deg}
    Suppose that $\mathcal C$ is a well-slicing family of a continuum $X \subset Q$.
    Then $\ph(X)$ is not a point.
  \end{thm}

  \begin{proof}
    Define $x \approx y$ whenever only finitely many elements of $\mathcal C$
    ray-separate $x$ and $y$.  Clearly, such a relation is symmetric and
    reflexive.  To see that it is transitive, suppose $x \approx y$ and $x \not
    \approx z$.  Then infinitely many elements of $\mathcal C$ ray-separate $x$
    and $z$.  However, only finitely many of these elements ray-separate $x$
    from $y$, and the rest then ray-separate $y$ from $z$, so $y \not \approx
    z$.

    Therefore, $\approx$ is an equivalence relation.  We will now show that $
    \approx$ is a closed equivalence relation by showing that $\{(x, y) \in Q^2
    \, \mid \, x \not \approx y\}$ is open.  Suppose that $x \not \approx y$.
    In particular, there are two elements $C_1$ and $C_2$ which ray-separate
    $x$ and $y$.  Every subspace of $\C$ is a normal space,
    so it is easy to see that sets $\widetilde{C_1}$ and $\widetilde{C_2}$
    separate $\overline{U_\iy}$ between every point $y$ in a
    neighborhood $V$ of $x$ and every point $z$ in a neighborhood $W$ of $y$.
    Then by Lemma~\ref{lem:no_finite} we can find infinitely many elements of
    $\mathcal C$ which do not contain $y$ or $z$ and separate $X$ between $C_1$
    and $C_2$. Each such element of $\mathcal C$ separates $X$ between $y$ and
    $z$ by Lemma~\ref{lem:middle_separates}. Hence no point in $V$ is
    $\approx$-equivalent to any point in $W$, and $\approx$ is closed. In
    particular, the partition of $Q$ into $\approx$-classes is upper
    semi-continuous.

    Now we show that, for any external angle $\alpha$, the impression
    $\imp(\alpha)$ is contained in a $\approx$-class.  To see this, suppose
    that $x, y \in \imp(\alpha)$ are ray-separated by two elements $B, C$ of
    $\mathcal C$. Since $B\cap C=\0$, we see that the set $\Theta(B)$ of angles
    associated with $B$ is disjoint from $\Theta(C)$. Hence at most one of
    these sets of angles contains $\al$, and we may assume that $\al\nin
    \Theta(C)$. Now, since $C$ is a ray-compactum, then each component $W$ of
    $\C\sm \wc$ corresponds to a well-defined open set of angles in $\uc$
    whose external rays are contained in $W$. Since $\al\nin \Theta(C)$, one
    such component $V$ contains $R_\al$ together with rays of close to $\al$
    angles. Hence $\imp(\al)\subset \ol{V}$ which means that $\imp(\al)$ is
    disjoint from all other components of $\ol{U_\iy}\sm \wc$ but $V$.
    However, by the assumption $C$ ray-separates $X$ between $x$ and $y$, hence
    the points $x\in \imp(\al)$ and $y\in \imp(\al)$ must belong to distinct
    components of $\ol{U_\iy}\sm \wc$, a contradiction.

    Finally, we show that $\ph(X)$ is not a point. First we refine $\approx$ to
    get an equivalence $\approx'$ with connected classes. Indeed, as in the
    proof of Lemma~\ref{lem:exists_partition} we can define a finer partition
    than that into $\approx$-classes whose elements are connected components of
    $\approx$-classes. Then the new partition is an upper semi-continuous
    monotone decomposition of $Q$ \cite[Lemma 13.2]{nad92}. By the previous
    paragraph any impression is still contained in an $\approx'$-class. Thus
    the quotient map $m:Q\to Q/\approx'$ is a monotone surjective map
    collapsing impressions. By Lemma~\ref{lem:collapse_impressions_lc} $Q/\approx'$ is
    locally connected. Now, let $C_1, C_2 \in \mathcal C$ be different.  For
    all $x \in C_1 \cap X$ and $y \in C_2 \cap X$, we see that $x \not \approx
    y$ by Lemma~\ref{lem:no_finite} and hence $m(x)\ne m(y)$.  Since $\ph$ is
    the finest monotone map, we see that $\ph(x) \neq \ph(y)$, and so $\ph(X)$
    is not a point.
  \end{proof}

  Now we prove a related criterion: If an unshielded continuum $Q
  \subset \C$ has an uncountable family of disjoint ray-continua, each
  of which ray-separate $Q$, then there is a sub-family which is
  well-sliced, and therefore the finest model is non-degenerate.

  \begin{lem}\label{lem:linear_order}
    Let $\mathcal C$ be an uncountable collection of disjoint ray-continua of
    an unshielded continuum $Q\subset\mathbb{C}$, each of which ray-separates
    $Q$. Then there exist elements $C_0, C_1 \in \mathcal C$ such that
    uncountably many elements of $\mathcal C$ ray-separate $C_0$ and $C_1$.
  \end{lem}

  \begin{proof}
    Assume by way of contradiction that this is not the case.  For $A, B \in
    \mathcal{C}$, let $Y_{AB}$ denote the set of points $x\in X\setminus(A \cup
    B)$ which are not ray-separated from $\A$ by $B$, nor vice-versa. We see
    that $Y_{AB}$ is an open subset of $X$, each $Y_{AB}$ contains every
    element of $\mathcal{C}$ that it intersects, and by assumption each
    $Y_{AB}$ may contain only countably many elements of $\mathcal{C}$. Then
    the open set $U=\bigcup_{A, B\in \mathcal{C}}Y_{AB}$ is an open subset of
    $Q$ of which $\{Y_{AB}\}_{A,B\in J}$ forms an open cover. Since $Q$ is
    second countable, countably many $Y_{AB}$ cover $U$. We therefore
    conclude that the set of elements of $\mathcal{C}$ contained in (or
    intersecting) $U$ is countable.

    Consider now any $D \in \mathcal{C}$ contained in $Q\setminus U$. By the
    definition of $U$, $D$ does not ray-separate any pair of elements in
    $\mathcal{C}$, so $U$ must lie in a component of
    $\C\setminus\widetilde{D}$.  Let $V_{D}$ denote a different component of
    $\C\setminus\widetilde{D}$. Notice that, for any $D,E\in \mathcal{C}$ such
    that $D\cup E\subset Q\sm U$, $V_{D}\cap V_{E}=\emptyset$, since any
    point in their intersection by definition belongs to $Y_{DE}\subset U$
    while $V_{D}\cup V_{E}\subset\mathbb{C}\setminus U$. Therefore,
    $\{V_{A}\,\mid\, A \in \mathcal{C},\,A \nsubseteq U\}$ is an uncountable
    collection of disjoint open subsets of $X$, contradicting that $X$ is a
    metric continuum.
\end{proof}

\begin{thm}
  Suppose that $\mathcal{C}$ is an uncountable collection of pairwise-disjoint
  ray-continua in an unshielded continuum $Q\subset\mathbb{C}$, each of which
  ray-separates $Q$.  Then a subcollection of $\mathcal{C}$ forms a
  well-slicing family of $Q$, and the finest model of $Q$ is non-degenerate.
\end{thm}

\begin{proof}
  By Lemma~\ref{lem:linear_order}, without loss of generality we may assume
  that there are elements $\alpha_{0},\alpha_{1}\in \mathcal{C}$ such that all
  other elements of $\mathcal{C}$ ray-separate $\alpha_{0}$ and $\alpha_{1}$.
  Clearly, a linear order $\prec$ is induced on $\mathcal{C}$, where $\beta
  \prec \gamma$ whenever $\beta$ ray-separates $\alpha_0$ and $\gamma$ (for if
  neither ray-separates the other from $\al_0$, one of them does not
  ray-separate $\al_0$ and $\al_1$).

  To each element $\al\in \mathcal{C}$ we can associate a chord $\ell_\al$ so
  that this collection of chords in the unit disk is uncountable and also
  linearly ordered. Hence there exists an element $\al_{1/2}$ such that both
  intervals $(\alpha_0,\alpha_{1/2})_\prec$ and $(\alpha_{1/2},\alpha_1)_\prec$
  in $\mathcal{C}$ are uncountable. By induction we can define $\al_q$ for any
  dyadic rational $q, 0<q<1$. Then the collection $\{\al_q\}$ with $q$ dyadic
  rational is a well-slicing family. By Theorem~\ref{th:wc-non-deg}, the finest
  model of $Q$ is non-degenerate.
\end{proof}


\section{The finest model for polynomial Julia sets is dynamical}\label{finmodpol}

Now we show that if $Q=J_P$ is a connected polynomial Julia set then the finest
map $\ph$ (which we always canonically extend onto the entire plane as
explained above) semiconjugates $P$ to a branched covering map $g:\hc
\rightarrow \hc$, which we call the \emph{topological polynomial}. Call
$\ph(J_P)$ the \emph{topological Julia set} (see the diagram on page~\pageref{eq:commdiag}). In Section~\ref{intro} by a
topological polynomial we understood the map $f_\sim$ induced by $\sim$ on the
quotient space of a lamination $\sim$; since it will always be clear whether we
deal with a topological polynomial considered on $J_\sim$ or we deal with its
canonic extension on the entire plane, our terminology will not cause
ambiguity. Recall that $\mathcal D_Q=\mathcal D$ is the finest among all upper
semi-continuous partitions whose elements are unions of impressions of $Q$, or,
as we have shown above, the family of all fibers of the finest map $\ph$
(K-sets).

We now give a transfinite method for constructing the finest closed
equivalence relation $\sim$ respecting a given collection of continua
$\mathcal A$.  To begin, let $\sim_0$ denote the equivalence relation
such that $x \sim_0 y$ if and only if $x$ and $y$ are contained in a
connected finite union of elements of $\mathcal A$.  Typically,
$\sim_0$ does not have closed classes, so $\sim$ makes more
identifications.  If an ordinal $\alpha$ has an immediate predecessor
$\beta$ for which $\sim_{\beta}$ is defined, we define $x \sim_\alpha
y$ if there exist finitely many sequences of $\sim_{\beta}$ classes
whose limits comprise a continuum containing $x$ and $y$.  (Here, the
limit of non-closed sets is considered to be the same as the limit of
their closures.)  In the case that $\alpha$ is a limit ordinal, we say
$x\sim_\alpha y$ whenever there exists $\beta < \alpha$ such that $x
\sim_\beta y$.  Notice that the sequence of $\sim_\alpha$-classes of a
point $x$ (as $\alpha$ increases) is an increasing nest of connected
sets, with the closure of each being a subcontinuum of its successor.
It is also apparent that $\sim_\alpha$-classes are contained in
$\sim$-classes for all ordinals $\alpha$.

Let us now show that $\sim = \sim_\Omega$ where $\Omega$ is the
smallest uncountable ordinal.  To see this, we first note that
$\sim_\Omega = \sim_{(\Omega+1)}$.  This is because the sequence of
closures of $\sim_\alpha$-classes containing a point $x$ forms an
increasing nest of subcontinua, no uncountable subchain of which can
be strictly increasing.  Therefore, all $\sim_\alpha$-classes have
stabilized when $\alpha = \Omega$.  This implies that $\sim_\Omega$ is
a closed equivalence relation, since the limit of
$\sim_\Omega$-classes is a $\sim_{(\Omega+1)}$-class, which we have
shown is a $\sim_\Omega$-class again.  Finally, $\sim_\Omega$
identifies elements of $\mathcal A$ to points and
$\sim_\Omega$-classes are contained in $\sim$-classes, so $\sim$ and
$\sim_\Omega$ coincide.

\begin{thm}\label{thm:ph_dynamic}
  For any $D \in \mathcal D$, $P(D) \in \mathcal D$.
\end{thm}
\begin{proof}
  We prove by transfinite induction that, for any ordinal $\alpha$,
  the image of a $\sim_\alpha$-class is again a $\sim_\alpha$-class.
  It is the case that the image of a $\sim_0$-class is again a
  $\sim_0$-class.  For instance, if $x \sim_0 y$ there is a finite
  chain of impressions containing them, and the image is a finite
  chain of impressions containing $P(x)$ and $P(y)$.  Furthermore, if
  $P(x)$ is contained in the class of $y$, there is a finite union $K$
  of impressions connecting them.  Since $P$ is an open map, the
  component of the $P^{-1}(K)$ containing $x$ is a finite union of
  impressions containing a preimage of $y$.

  Now assume for induction that, for every $\beta<\alpha$, that
  $\sim_\beta$-classes map to other $\sim_\beta$-classes.  We will
  show that this is also true for $\sim_\alpha$-classes.  This is easy
  to see if $\alpha$ is a limit ordinal, so we will concentrate on
  proving this fact when $\alpha$ has an immediate predecessor
  $\beta$.  Suppose first that $x \sim_\alpha y$.  Then there are
  sequences $(K_i^1)_{i=1}^\infty \to K_1$, \ldots,
  $(K_i^n)_{i=1}^\infty \to K_n$ of $\sim_\beta$-classes such that
  $K_1$, \ldots, $K_n$ form a chain from $x$ to $y$ (i.e., so $K_1$
  contains $x$, $K_n$ contains $y$, and $K_i \cap K_{i+1} \neq
  \emptyset$ for $1 \le i < n$).  The image sequences
  $(P(K_i^1))_{i=1}^\infty \to P(K_1)$, \ldots,
  $(P(K_i^n))_{i=1}^\infty \to K_n$ are also sequences of
  $\sim_\beta$-classes by the inductive hypothesis, which converge
  onto the chain $K_1 \cup \ldots \cup K_n$.  This illustrates that
  $P(x) \sim_\alpha P(y)$.

  On the other hand, say $P(x) \sim_\alpha y$; we will show that $x
  \sim_\alpha z$ for some $z \in P^{-1}(y)$, so that the
  $\sim_\alpha$-class of $x$ maps \emph{onto} to $\sim_\alpha$-class
  of $P(x)$.  Again find sequences $(K_i^j)_{i=1}^\infty$ of
  $\sim_\beta$-classes with $j \in \{1, \ldots, n\}$ whose limits
  $K_1$, \ldots, $K_n$ form a chain from $P(x)$ to $y$.  Because $P$
  is open, there is a sequence of preimages of $(K_i^1)_{i=1}^\infty$,
  (whose members are $\sim_\beta$-classes by hypothesis) that limit to
  a continuum $K_1'$ containing $x$.  By continuity, $P(K_1') = K_1$
  intersects $K_2$, so we can proceed by inductively choosing limits
  $K_{i+1}'$ of $\sim_{\beta}$-classes intersecting $K_i'$ and mapping
  onto $K_{i+1}$.  The resulting chain $K_1'\cup \ldots\cup K_n'$ maps
  onto $K_1 \cup \ldots \cup K_n$, which shows that $x$ is
  $\sim_\alpha$-equivalent to some preimage of $y$. We therefore see
  that $\sim_\alpha$-classes map onto $\sim_\alpha$-classes, by
  letting $\alpha = \Omega$ that elements of $\mathcal D$ map onto
  elements of $\mathcal D$.
\end{proof}

  The next theorem follows from Theorem~\ref{thm:ph_dynamic}.

  \begin{thm}\label{thm:model_map}
    The map $\ph$ semiconjugates $P$ to a branched covering map \linebreak $g:\C \rightarrow \C$.
  \end{thm}

  \begin{proof}
    Let $m:\C \rightarrow \C$ be the quotient map corresponding to
    $\mathcal D$ as constructed in Lemma~\ref{lem:monotone_maps_extend}.
    The map $m \circ P : \C \rightarrow \C$ is continuous, and is
    constant on elements of $\mathcal D$ by
    Theorem~\ref{thm:ph_dynamic}.  Therefore, there is an induced
    function $g:\C \rightarrow \C$ such that $m \circ P = g \circ
    m$. Also, it is easy to see that $g$ is continuous. Indeed, let
    $x_i\to x$. Then $\ph^{-1}(x_i)$ converge into $\ph^{-1}(x)$ and
    $P(\ph^{-1}(x_i))$ converge into $P(\ph^{-1}(x))$. Applying
    $\ph$ to this, we see that $g(x_i)=\ph(P(\ph^{-1}(x_i)))$
    converge to $g(x)=\ph(P(\ph^{-1}(x)))$, and so $g$ is
    continuous.

    To see that $g$ is open, let $U \subset \C$ be an open set. Then
    $m^{-1}(U)$ is an open set. By the previous paragraph,
    $P(m^{-1}(U))=m^{-1}(g(U))$. Now by Theorem~\ref{thm:ph_dynamic}
    and by the definition of a quotient map $m(P(m^{-1}(U))) = g(U)$
    is open. Since $U$ was arbitrary, $g$ is an open map. By the
    Stoilow Theorem \cite{sto56} $g$ is branched covering.
  \end{proof}

  In what follows we always denote by $g$ the topological polynomial to which
  $P$ is semiconjugate; the $\ph$-image of $J_P$ is denoted by $J_{\sim_P}$.
  Define $\Phi:\ucirc \rightarrow J_{\sim_P}$ as $\Phi = \ph \circ \imp$. From
  Lemma~\ref{lem:collapse_impressions_lc}, $\Phi$ is a well-defined continuous
  function.

  \begin{thm}\label{thm:lamination_model}
    The map $\Phi$ semiconjugates $z \mapsto z^d$ to $g|_{J_{\sim_P}}$.
  \end{thm}

  \begin{proof}

    Define $\sigma_d = z \mapsto z^d$.  Recall that $g$ is defined so that
    (1) $g \circ \ph = \ph \circ P$ and also that the B\"ottcher uniformization gives that
    (2) $P \circ \imp = \imp \circ \sigma_d$. We then see that, as desired,

    \begin{align*}
      g \circ \Phi &= g \circ \ph \circ \imp \\
      &= \ph \circ P \circ \imp & \text{(by (1))}\\
      &= \ph \circ \imp \circ \sigma_d & \text{(by (2))}\\
      &= \Phi \circ \sigma_d.
    \end{align*}
  \end{proof}

  Then, as in the previous section, the \emph{finest lamination} corresponding
  to $J_{\sim_P}$ is the equivalence relation $\sim_P$ on $\ucirc$, defined by
  $\alpha_1 \sim_P \alpha_2$ if and only if $\Phi(\alpha_1) = \Phi(\alpha_2)$.

\section{A criterion for the polynomial Julia set to have
a non-degenerate finest monotone model}\label{criter}

Here we obtain the remaining main results of the paper. We give a necessary and
sufficient condition for the existence of a non-degenerate locally connected
model of the connected Julia set of a polynomial $P$. We give this criterion in
terms of its rational lamination as well as the existence of specific wandering
continua in the Julia set behaving in the fashion reminiscent of the irrational
rotation on the unit circle.

\subsection{Topological and laminational preliminaries} Let us recall the following
definitions. A finite set $A$ is said to be \emph{all-critical} if
$\si(A)$ is a singleton. A finite set $B$ is said to be
\emph{eventually all-critical} if there exists a number $n$ such
that $\si^n(B)$ is a singleton. The following result is obtained in
\cite{bfmot10}, Theorem 7.2.7.

\begin{thm}\label{fixpts}
  Suppose that $J_P$ is the connected Julia set of a polynomial $P$ such that
  its locally connected model $J_\sim$ corresponding to the lamination
  $\sim=\sim_P$ is a dendrite. Then there are infinitely many periodic
  cutpoints of $J_\sim$ and, respectively, $\sim$-classes, each of which
  consists of more than one point.
\end{thm}

We will also need another result from \cite{bfmot10}; recall that by
$K_P$ we denote the filled-in Julia set of a polynomial $P$.

\begin{thm}\label{degimp}
Suppose that $P:\Complex\to\Complex$ is a polynomial, $X\subset K_P$
is a non-separating continuum or a point such that $P(X)\subset X$,
all fixed points in $X$ are repelling or parabolic, for every Fatou
domain $U$ of $P$ either $U\subset X$ or $U\cap X=\0$, and for each
fixed point $x_i\in X$ there exists an external ray $R_i$ of $X$,
landing at $x_i$, such that $P(R_i)=R_i$. Then $X$ is a single
point.
\end{thm}

Theorem~\ref{fixpts} applies in the proof of Theorem~\ref{neces}. Define an
\emph{all-critical} point as a cutpoint of $J_\sim$ whose image is an endpoint
of $J_\sim$.

\begin{thm}\label{neces}
  Suppose $\sim$ is a lamination. Then at least one of the following
  properties must be satisfied:

  \begin{enumerate}

  \item $J_\sim$ contains the boundary of a parattracting Fatou domain;

  \item there are infinitely many periodic $\sim$-classes each of
    which consists of more than one angle;

  \item there exists a finite collection of all-critical
    $\sim$-classes with pairwise disjoint grand orbits whose images under the
    quotient map form the set of all-critical points on the boundaries of
    Siegel domains from one cycle of Siegel domains so that \emph{all}
    cutpoints of $J_\sim$ on the boundaries of these Siegel domains belong to
    the grand orbits of these all-critical points.
  \end{enumerate}

\end{thm}

\begin{proof} Suppose that $J_\sim$ is a dendrite. Then the result follows from
Theorem~\ref{fixpts}. Suppose now that $J_\sim$ is not a dendrite. Then
$J_\sim$ contains a simple closed curve $S$. By Lemma~\ref{two-dyns}, there are
two cases possible. First, we may assume that $S$ is the boundary of a periodic
parattracting Fatou domain. Then (1) holds.

Consider the case when $S$ is of period $1$ and $f_\sim|_S$ is conjugate to an
irrational rotation (the case of higher period is similar). Consider a point
$x\in S$ which is a cutpoint of $J_\sim$ ($x$ must exist since $S\ne J_\sim$).
Then $x$ is not (pre)periodic. Hence by Theorem~\ref{kiwi-wan} the $\sim$-class
corresponding to $x$ is finite. This implies that the number of components of
$J_\sim \sm \{x\}$ is finite. One such component contains $S\sm \{x\}$ while
all others have closures intersecting $S$ exactly at $x$. Denote by $B_x$ the
union of $x$ and all such components not containing points of $S$. Clearly the
set $B_x$ is closed and connected.

Let us show that $x$ is eventually mapped into a point which is not a cutpoint
of $J_\sim$. Indeed, otherwise all points $f_\sim^i(x)$ are cutpoints of
$J_\sim$. Since there are finitely many critical points of $f_\sim$, we can
then choose $N$ such that no set $B_{f_\sim^m(x)}$ contains a critical point
for $m\ge N$. On the other hand, $f_\sim^N(x)$ is a cutpoint of $J_\sim$ by the
above. Hence $B_{f_\sim^N(x)}$ is a wandering continuum in $J_\sim$, a
contradiction with Theorem~\ref{nowanco}. Now the connection between
$\sim$-classes and points of $J_\sim$ implies that the $\sim$-class
corresponding to $x$ is eventually all-critical. Clearly, any all-critical
point $y\in S$ corresponds to an all-critical $\sim$-class which meets the
boundary of the corresponding Siegel domain $U$ of $\sim$ in a leaf (since
$\sim$-classes are convex). Moreover, an all-critical point in $S$ is a
cutpoint of $J_\sim$ whereas all forward images of an all-critical point are
endpoints of $J_\sim$. Hence the all-critical classes which are non-disjoint
from $\ol{U}$ have pairwise disjoint grand orbits. Clearly, this implies the
properties listed in the case (3) of the theorem. Similar arguments go through
if $S$ is periodic rather than fixed.
\end{proof}

By Theorem~\ref{main2}, if the finest model is not degenerate then it gives
rise to a non-degenerate lamination $\sim_P$. Hence one of the three phenomena
described in Theorem~\ref{neces} will have to take place in $J_{\sim_P}$.
Thanks to the existence of the finest map, this implies that corresponding
phenomena will be present in the Julia set $J_P$. In other words, the presence
of at least one of the phenomena is a \emph{necessary} condition for the
existence of a non-degenerate finest model (we will formalize this observation
later on in Theorem~\ref{csgap}). However now our main aim is to show that the
presence of at least one of the phenomena is \emph{sufficient} for the
existence of a non-degenerate finest model. The main tool here is well-slicing
studied in Subsection~\ref{well-crit}. We will describe three cases in which we
establish sufficient conditions for the existence of well-slicing for the Julia
set (and hence, by the results from Subsection~\ref{well-crit}, for the
non-degeneracy of its finest model). The sufficient conditions are stated in a
step by step fashion in a series of lemmas and propositions.

\subsection{The case of infinitely many periodic cutpoints}

Next we want to suggest a sufficient condition for the non-collapse of the
entire $J_P$ corresponding to the case (2) of Theorem~\ref{neces}. However this
time we need a lot of preparatory work. First we study CS-points and CS-cycles
(recall that a \emph{CS-point} is a Siegel or Cremer periodic point and a
\emph{CS-cycle} is a cycle of CS-points). Call a set $X$ \emph{periodic (of
period $m$)} if $X, \dots, P^{m-1}(X)$ are pairwise disjoint while
$P^m(X)\subset X$. Then the union $\cup^{m-1}_{i=0} P^i(X)$ is said to be a
\emph{cycle of sets} (we can then talk about \emph{cycles of continua} and the
like).

\begin{lem}\label{hedgehog}
  If $Y$ is a cycle of continua containing a CS-cycle and a periodic point not
  from this CS-cycle then it contains a critical point of $P$.
\end{lem}

\begin{proof}
  We only consider the case when $Y$ is an invariant continuum; the case of the
  cycle of continua can be dealt with similarly.  Suppose that $Y$ contains no
  critical points. Choose a neighborhood $U$ of $Y$ such that no critical
  points belong to $\ol{U}$, consider the set of all points never exiting
  $\ol{U}$, and then the component $K$ of this set containing the given
  CS-point $p$; clearly, $Y\subset K$. Such sets are called \emph{hedgehogs}
  (see \cite{pere94,pere97}) and have a lot of important properties. In
  particular, $K$ cannot contain any other periodic points. On the other hand,
  $Y\subset K$, a contradiction with the assumption that there is a periodic
  point in $Y$, distinct from $p$.
\end{proof}

Next we prove a few lemmas which discuss properties of $J_P$ at (pre)periodic
points. We need them for two reasons. First of all, they help us establish the
next sufficient condition for the non-collapse of $J_P$ under the finest map.
Second, they give sufficient conditions on a (pre)periodic point to be a point
of local connectivity of the Julia set. In that sense they generalize Kiwi's
theorem \cite{kiwi97} where he proves (using different methods) that in the
absence of CS-points the Julia set is locally connected at its (pre)periodic
points.

There are two competing laminations which both reflect the structure of $J_P$,
the rational lamination $\approx_{rat}$ and the finest lamination $\sim_P$. We
use both of them to suggest sufficient conditions for $J_P$ to be locally
connected at a (pre)periodic point $p$. Recall that $A(y)$ is the set of all
angles whose rays land at a point $y\in J_P$.

\begin{lem}\label{lcpt} Suppose that $K=\imp(A)$ is the union of
  impressions of a finite set of periodic angles $A$ which is periodic,
  connected and disjoint from impressions of all other angles. Then $K$ is a
  repelling or parabolic periodic point.  Thus, if $p$ is a (pre)periodic
  point  of $P$ and $\Phi^{-1}(\ph(p))$ is finite then $\ph^{-1}(\ph(p))=\{p\}$ (i.e.,
  $\{p\}$ is a K-set) and $J_P$ is locally connected at $p$.
\end{lem}

\begin{proof}
  It is easy to see that if $K$ contains a point of a Fatou domain
  in its topological hull, then the entire Fatou domain is contained
  in this topological hull. Now, if $K$ contains a parattracting
  Fatou domain in its topological hull, then infinitely many
  periodic repelling points on its boundary (which exist by
  \cite{pz94}) would give rise to infinitely many impressions
  non-disjoint from $K$, a contradiction. Let us show that the
  topological hull $\tl(K)$ of $K$ cannot contain a CS-point either.
  Indeed, otherwise by Lemma~\ref{hedgehog} it has to contain a
  critical point $c$.  Moreover, since $\tl(K)$ does not contain
  parattracting Fatou domains, $c\in J_P$. Then the symmetry around
  critical points implies that there are two angles in $A$ which map
  into one (recall that the only angles whose impressions may
  contain $c$ are the angles of $A$), in contradiction with the fact that angles in $A$ were periodic.
  Thus $K$ is an invariant non-separating
  subcontinuum of $J_P$ which contains no CS-points.  On the other
  hand, by the assumptions, there are only finitely many periodic
  points in $K$ (because $K$ is disjoint from impressions of all
  angles except for finitely many). Hence all of these points are
  repelling or parabolic and together with the rays landing at them
  may be assumed to be fixed.  By Theorem~\ref{degimp} this implies
  that $K$ is a repelling or parabolic periodic point.

  To establish the next implication of the lemma, assume that $p$ is a
  $P$-periodic point and $\ph(p)=x$. We may assume that $x$ is $g$-fixed. Set
  $A=\Phi^{-1}(x)$.  By the assumptions of the lemma $\Phi^{-1}(\ph(p))$ is
  finite. Hence we may assume that all angles in $A$ are fixed. Clearly,
  $B=\ph^{-1}(x)$ is an invariant continuum. By the assumptions only angles of
  $A$ can have impressions non-disjoint from $B$, and there are finitely many
  of them. Hence by the first part of the lemma $B$ is a repelling or parabolic
  periodic point. The remaining claim that $J_P$ is locally connected at $p$
  now follows from Lemma~\ref{lc}.
\end{proof}

The next lemma relies upon Lemma~\ref{lcpt}. Recall, that by
$\approx_{rat}$ we denote the finest lamination which respects the
geometric lamination $\lam_{rat}$. Properties of $\approx_{rat}$ are
studied in Theorem~\ref{periodic-prel} (in particular, it is shown
there that $\approx_{rat}$ is not degenerate). Recall that gaps of a
lamination understood as an equivalence class of an equivalence
relation are normally denoted by a small boldface letter (such as
$\g$) while geometric gaps of geometric laminations are normally
denoted by capital letters (such as $G$). Also, recall that $R$ is
the set of all periodic repelling (parabolic) bi-accessible points
and their preimages.

\begin{lem}\label{lcpt1}
  Suppose that $\g$ is a (pre)periodic finite gap or leaf of $\approx_{rat}$
  disjoint from boundaries of Fatou domains of $\approx_{rat}$. Then $\ch(\g)$
  is a gap or leaf of $\lam_{rat}$, coinciding with the set $A(p)$ for a point
  $p\in R$, and $J_P$ is locally connected at $p$.
\end{lem}

\begin{proof}
  Assume that $\g$ is invariant. Since no leaf of $B=\bd(\ch(\g))$
  can come from the boundary of a Fatou domain of $\approx_{rat}$, all leaves
  in $B$ are limit leaves of $\lam_{rat}$. The upper semi-continuity of
  impressions now implies that the union of impressions $\imp(\g)$ of angles of
  $\g$ is a continuum itself. Moreover, it is disjoint from impressions of all
  angles not in $\g$ because for any such angle $\ga$ we can find a leaf of
  $\lam_{rat}$ which cuts $\imp(\ga)$ off $\imp(\g)$. Now the lemma follows
  from Lemma~\ref{lcpt}.
\end{proof}

We need another preparatory result, dealing with laminations generated by
collections of periodic gaps and leaves. If we fix a set $A$, then a set
$B\subset A$ is said to be \emph{cofinite (in $A$)} if $|A\sm B|$ is finite.
Given a generating (and hence invariant) family $A$ of pairwise disjoint periodic gaps or leaves we
then consider a geometric prelamination $\lam_A$ consisting of $A$ and
preimages of elements of $A$. By Theorem~\ref{periodic-prel} we can construct
the corresponding lamination $\approx_A$; given a finite gap or leaf $G$ from
$\lam_A$, there exists a finite $\approx_A$-class $\cl_A(G)$ containing $G'$ and
called the $\approx_A$-class \emph{generated} by $G$. Denote by $p_A$ the
quotient map from $\uc$ to $J_{\approx_A}$.

\begin{lem}\label{fincut} Suppose that $A$ is an infinite
generating family of periodic gaps or leaves under the map $\si_d$.
Then there exists a cofinite invariant subset $D'\subset A$ such
that any cofinite invariant set $E\subset D'$ has the following
properties.

\begin{enumerate}

\item If $G$ is a leaf or finite gap of $\lam_{\approx_E}$ then
$\cl_E(G)\cap \bd(U)=\0$ for any Fatou domain $U$ in
$\lam_{\approx_E}$ or $\lam_{\approx_A}$.

\item The family of $\approx_E$-classes generated by the elements of $\lam_E$
is a well-slicing family of $\uc$.


\end{enumerate}

\end{lem}

\begin{proof}
To prove claim (1) suppose that there is a finite invariant collection $Q$ of
elements of $A$ for which there exists a leaf $\ell=ab$ of
$\lam_A$ or a point $x\in \uc$ with the following properties:
\begin{enumerate}

\item $\ell$ (resp., $x$) is a limit leaf for a sequence of
leaves of $\lam_A$ with endpoints in the positively oriented arc
$(a, b)$ (resp., with endpoints on both sides of $x$);

\item there is $\e>0$ such that any leaf of $\lam_A$ with
endpoints in $(a, a+\e)$ and $(b-\e, b)$ (resp., in $(x-\e, x+\e)$)
is a preimage of a boundary leaf of a set from $Q$.
\end{enumerate}
Then we call $Q$ a \emph{finite limiting collection} (of elements of $A$).

Let us proceed as follows. Suppose that there exist no finite
limiting collections. Take a cofinite invariant set $E\subset A$.
Let us show that the Fatou gaps of $\approx_E$ coincide with the
Fatou gaps of $\approx_A$. Indeed, otherwise there is a Fatou gap
$G$ of $\approx_E$ which was not a Fatou gap of $\approx_A$.

Clearly, gaps from the orbit of $G$ contain no sets from $E$.
However, gaps from the orbit of $G$ must contain some sets from
$A\sm E$ (otherwise these gaps would be gaps of $\approx_A$ as
well). Denote the sets from $A\sm E$ contained in gaps from the
orbit of $G$ by $T_1, \dots, T_r$. There must exist a leaf $\ell$ of
$\lam_{\approx_A}$ or a point $x\in \uc$ such that the leaves from the grand
orbits of sets $T_1, \dots, T_r$ contained inside gaps from the
orbit of $G$ accumulate on one side of $\ell$ (resp., at the point
$x$ while separating $\disk$ in two components the smaller of which
contains $x$). Moreover, only sets $T_1, \dots, T_r$ can have leaves
accumulating upon $\ell$ (resp., $x$) in this way (because only
these sets from $A$ are contained in gaps from the orbit of $G$).
This implies that $T_1, \dots, T_r$ is a finite limiting collection,
a contradiction.

Now, suppose that there exists a finite limiting collection $Q_1$
and $\ell$ is a limit leaf of $Q_1$ existing by the definition
(the case of a point $x$ is considered similarly). Remove $Q_1$ from
$A$ and consider a generating set $E_1=A\sm Q_1$ and the
corresponding laminations $\approx_1$ and $\lam_{\approx_1}$. It
follows that there is a gap $G_1$ of $\lam_1$ with the boundary leaf
$\ell$ located on the same side of $\ell$ from which the pullbacks
of leaves of sets from $Q_1$ approach $\ell$ in $\lam_A$. Clearly,
$G_1$ cannot have boundary leaves concatenated to $\ell$ at $\ell$'s
endpoints for if such boundary leaves exist, they will have to be
leaves of $\lam_{\approx_A}$ too which is impossible because then
they would intersect the leaves from the grand orbits of sets from
$Q_1$ which converge to $\ell$ from the appropriate side by the
assumption. Hence, $G_1$ is a Fatou gap of $\lam_{\approx_1}$ which
did not exists in $\lam_{\approx_A}$.

We repeat this construction over and over until, after finitely many
steps, we will find a cofinite invariant subset of $A$ which has no
finite limiting collections. Indeed, in the process of finding sets
$E_1\supset E_2\supset \dots$, on each step at least one new Fatou gap $G_1,
G_2, \dots$ appears. Clearly, at each step all the Fatou gaps of the
current lamination $\lam_k$ allow us to draw a maximal collection of
pairwise disjoint critical leaves inside them, and the number of
critical leaves in such a collection is bounded by $d-1$. Hence the
number of such critical leaves eventually stabilizes which implies
that from this moment on the new Fatou gaps will have to contain the
previously existing ones. This implies that the periods of the new
Fatou gaps can only be smaller than the periods of the ones which
had existed before. Therefore, from some time on the appearance of
new infinite gaps becomes impossible.

Denote the corresponding cofinite invariant subset $E_m$ of $A$ by
$D$. By the construction $D$ has no finite limiting collections,
hence by the first paragraph of the proof all its cofinite invariant
subsets $S$ generate a lamination $\approx_S$ which has the same
Fatou gaps as $\approx_D$. Each periodic Fatou domain of $\approx_D$
has finitely may $\approx_D$-gaps/leaves non-disjoint from its
boundary. Denote by $D'\subset D$ the family of all \emph{other}
$\approx_D$-gaps/leaves (observe that $D'$ is cofinite and hence
infinite). Suppose now that $E\subset D'$ is cofinite. Then by the
above the Fatou gaps of $\approx_E$ coincide with Fatou gaps of
$\approx_D$, and by the choice of $E$ no element of $\lam_E$
intersects $\bd(U)$ where $U$ is a Fatou gap of $\lam_{\approx_E}$.
Moreover, since $E\subset A$, $\ol{\lam_E}\subset \ol{\lam_A}$.
Hence, Fatou domains of $\lam_{\approx_A}$ are contained in Fatou
domains of $\lam_{\approx_E}$. This proves claim (1) of the lemma.

To prove claim (2), let $G, H\in \lam_E$ are such that $\cl_E(G)$
and $\cl_E(H)$ are distinct. Suppose that there are no
$\approx_E$-classes, generated by elements of $\lam_E$, separating
$\uc$ between $\cl_E(G)$ and $\cl_E(H)$. Since by the construction
$\approx_E$-classes generated by elements of $\lam_E$ are dense in
$\lam_{\approx_E}$, there must be a Fatou gap of $\lam_{\approx_E}$
on whose boundary both $\cl_E(G)$ and $\cl_E(H)$ lie which is
impossible by the above. This completes the proof of (2).
\end{proof}


\begin{prop}\label{biacc}
Suppose that $p\in J_{\sim_P}=\ph(J_P)$ is a periodic point such that $\Phi^{-1}(p)$
is infinite. Then there are no more than finitely many periodic leaves of the
rational geometric lamination $\lam_{rat}$ connecting angles of $\Phi^{-1}(p)$. In
particular, {\rm (1)} the set of all bi-accessible periodic repelling or
parabolic points in $\ph^{-1}(p)$ must be finite, and {\rm (2)} if the set of
\textbf{all} repelling bi-accessible periodic points of $P$ is infinite then
the finest model is non-degenerate.
\end{prop}

\begin{proof}
  We may assume that $p$ is a fixed point of $g$; then $\Phi^{-1}(p)$ is an
  infinite gap of $\sim_P$. Set $G=\Int(\ch(\Phi^{-1}(p)))$; by
  Lemma~\ref{fat-gap-domain} $G$ is a Fatou gap of $\lam_{\sim_P}$ and hence by
  Lemma~\ref{good-gap} there is a monotone semiconjugacy $\psi$ of
  $\si^*|_{\bd(G)}$ and a map $\si_k:\uc\to \uc$ with the appropriately chosen
  $k>1$. The map $\psi$ collapses all chains of concatenated leaves in $\bd(G)$
  to points; by Lemma~\ref{lem:periodic_gap} all leaves in the chains are
  (pre)periodic and by Lemma~\ref{good-gap} and Lemma~\ref{fat-gap-domain} each
  chain consists of at most $N$ leaves ($N$ depends on $G$). By way of
  contradiction suppose that there are infinitely many periodic leaves of the
  rational prelamination $\lam_{rat}$ connecting angles of $\Phi^{-1}(p)$. The
  idea is to use the map $\psi$ in order to transport the restriction of
  $\lam_{rat}$ onto $\Phi^{-1}(p)$ to the entire circle $\uc$, then to use
  Lemma~\ref{fincut} to find a well-slicing family of $\uc$ consisting of
  (pre)periodic geometric gaps and leaves of $\uc$ corresponding to elements of
  $\lam_{rat}$, and then to show that ray-continua corresponding to those elements
  of $\lam_{rat}$ form a well-slicing family of $\ph^{-1}(p)$. By
  Theorem~\ref{th:wc-non-deg} then $\ph(\ph^{-1}(p))$ is not a point, a
  contradiction.

  The leaves of $\lam_{rat}$ which lie in the boundary of
  $\ch(\Phi^{-1}(p))=G$ will produce just points under
  $\psi$. However, by Lemma~\ref{lem:periodic_gap} there are only
  finitely many periodic leaves in $\bd(G)$. Hence by the assumptions
  of the proposition there are infinitely many periodic geometric
  leaves or gaps of $\lam_{rat}$ contained in $G$ and such that $\psi$
  does not identify points of their bases with other points at
  all. Denote their family by $A$; also, denote the family of all
  their preimages under all powers of $\si$ \emph{contained in $G$} by
  $\widehat{\lam}_A$ (recall, that the notation $\lam_A$ is reserved
  for the collection of \emph{all} preimages of elements of
  $A$). Thus, $\widehat{\lam}_A$ is the family of all (pre)periodic
  geometric leaves and gaps of $\lam_{rat}$ contained in $G$ and not
  in $\bd(G)$. Define the geometric prelamination
  $\lam'=\psi(\widehat{\lam}_A)$ on the entire circle $\uc$ as the
  family of convex hulls of $\psi$-images of bases of elements of
  $\widehat{\lam}_A$ (recall that $\psi$ is defined only on
  $\bd(G)$). It is easy to see that this indeed creates a geometric
  prelamination whose all leaves are (pre)periodic. By the choice of
  $A$ in this way each gap/leaf of $\lam=\widehat{\lam}_A$ is
  transported by $\psi$ to the corresponding gap/leaf of $\lam'$ in a
  one-to-one fashion. Then $\psi(A)$ is the family of periodic
  geometric leaves and gaps of $\lam'$. Clearly, $\psi(A)$ is infinite
  and the lamination $\lam'$ is the same as the lamination
  $\lam_{\psi(A)}$ introduced right before Lemma~\ref{fincut} in which
  appropriate preimages of elements of $\psi(A)$ are used.

  By Lemma~\ref{fincut} there exists a cofinite family $B\subset A$
  satisfying both properties listed in Lemma~\ref{fincut}. In
  particular, as in Lemma~\ref{fincut} for $B$ the prelamination
  $\lam_B$ and the corresponding lamination $\approx_B$ can be
  constructed. By the choice of $A$ the map $\psi$ then allows us to
  pull them back to $G$ in a one-to-one fashion and without changing
  the order. Now, by claim (1) of Lemma~\ref{fincut} if $h\in \lam_B$
  then $\cl(h)\cap \bd(U)=\0$ for any Fatou domain $U$ in
  $\lam_{\approx_B}$ (here $\cl(h)$ is understood in the sense of the
  lamination $\approx_B$, i.e. $\cl(h)$ is the $\approx_B$-class containing
  $h$). Let us show that then in fact $h=\cl(h)$ and $\psi^{-1}(h)\in
  \lam_{rat}$. Indeed, consider leaves on the boundary of $\cl(h)$. By
  Theorem~\ref{periodic-prel} they all are limits of leaves of
  elements of $\lam_B$. It follows that then leaves on the boundary of
  $\psi^{-1}(\cl(h))$ are limit leaves for leaves of $\psi$-preimages
  of elements of $\lam_B$. Thus, leaves on the boundary of
  $\psi^{-1}(\cl(h))$ are limit leaves for leaves of
  $\lam_{rat}$. This implies that the impression of any angle not from
  $\psi^{-1}(\cl(h))$ is cut off $\imp(\psi^{-1}(\cl(h)))$ by tails of
  the appropriate points of $R$ and hence is disjoint from
  $\imp(\psi^{-1}(\cl(h)))$. By Lemma~\ref{lcpt} then $h=\cl(h)$ and
  $\psi^{-1}(h)\in \lam_{rat}$.

  Now, by Lemma~\ref{fincut} $\lam_B$ is a well-slicing family of
  $\uc$. By the previous paragraph and by the properties of the map
  $\psi$ it follows that the family of degenerate ray-continua
  $\imp(\psi^{-1}(h)), h\in \lam_B$ is a well-slicing family of
  $\ph^{-1}(p)$ and hence by Theorem~\ref{th:wc-non-deg}
  $\ph(\ph^{-1}(p))$ is not a point, a contradiction. This proves
  (1). Now, if the finest model is degenerate then the degenerate
  topological Julia set can play the role of the point $p$, the entire
  circle $\uc$ plays the role of the $\approx_P$-class $\Phi^{-1}(p)$,
  and (1) implies that $R$ is finite. Hence, (2) follows and the proof
  is completed.
\end{proof}

\subsection{The Siegel case}

Now we establish the third sufficient condition for the non-degeneracy of
the finest model, this time corresponding to the case (3) of
Theorem~\ref{neces}. However first we need to introduce the appropriate
terminology.

As was explained in Section~\ref{lamprel}, the closure of any invariant
geometric prelamination is a geometric lamination. This idea was used when the
geo-lamination $\ol{\lam_{rat}}$ was constructed. However it can also be used
in other situations. Suppose that there exists a finite collection $\Kc$ of
wandering ray-continua $K_i, i=1, \dots, m$. We will call $\Kc$ a
\emph{wandering collection} if distinct forward images of continua $K_i$ are
all pairwise disjoint. By the arguments similar those from Theorem 4.2
\cite{bo08a} one can associate to $\Kc$ a geometric prelamination $\lam_{\Kc}$
\emph{generated by $\Kc$}, and then its closure - a geo-lamination
$\ol{\lam_{\Kc}}$ \emph{generated by $\Kc$}. For completeness we will briefly
explain the main ideas of this theorem.

First we need to construct the grand orbit of sets from $\Kc$. However it may
happen that simply taking pullbacks of the forward images of these sets will
lead to their growth. Indeed, suppose, for example, that $K_1$ contains a
critical point $c$. Then already the first pullback of $P(K_1)$ may well be
bigger than $K_1$. If as we iterate the map $K_1$ hits several critical points,
the same can take place several times. However since $\Kc$ is a wandering
collection we can choose a big $N$ so that the continua $P^N(K_i), i=1, \dots, m$
are non-precritical.

If we now take these continua, all their forward images, and then all pullbacks
of these forward images we will get a ``consistent'' grand orbit of several
sets meaning that for every set $Q$ from the grand orbit in question the
$P^i$-pullback of $P^i(Q)$ containing $Q$ coincides with $Q$. As a result of
the construction the initially given ray-continua may have grown, however they
will have (eventually) the same images as the originally given continua. In
particular, the continua $K_i$ may have grown to new continua $K'_i$, and we
can think of the just constructed grand orbit $\Ga$ as the grand orbit of the
family of continua $K'_1, \dots, K'_m$. Observe that $\Kc'=\{K'_1, \dots,
K'_m\}$ is still a wandering collection. Hence, since all Fatou domains are
(pre)periodic, any set from $\Ga$ is a non-separating subcontinuum of $J_P$.

Since each $K_i$ is a ray-continuum, by Definition~\ref{raycont} there is a set
of angles associated to $K_i$ in that the union of the principal sets of these
angles is contained in $K_i$ while the union of their impressions contains
$K_i$. The new continuum $K'_i$ is obtained as the union of $K_i$ with some
pullbacks of its images. Hence and because the collection of all principal sets
and impressions is invariant we see that $K'_i$ is also a ray-continuum. It
follows that in fact any continuum $K'\in \Kc'$ is a ray-continuum, and if we
define the set of angles $\Ta(K')=H_{K'}$ as the set of all angles whose
principal sets are contained in $K'$ then we will have

\[
    \bigcup_{\theta \in H_{K'}}\acc(\theta) \subset K' \subset
    \bigcup_{\theta \in H_{K'}}\imp(\theta)
    \]
which means that the set of angles $H_{K'}$ is associated with the continuum
$K'$ in the sense of the Definition~\ref{raycont}. Observe that by
Theorem~\ref{kiwi-wan} the sets of angles $H_{K'}, K'\in \Kc'$ cannot have more
than $2^d$ angles (and therefore they are closed).

Now it is not hard to show (see Theorem 4.2 \cite{bo08a}) that the family of
convex hulls of so defined sets of angles $H_{K'}, K'\in \Ga$ form a geometric
prelamination which we denote by $\lam_\Kc$. By the definition each original
ray-continuum $K_i$ has the associated to it set of angles $A_i$, and it
follows that $A_i\subset H_{K'_i}=H_i$. Therefore each $A_i$ is contained in a
leaf or gap of $\lam_\Kc$. Then the closure $\ol{\lam_{\Kc}}$ of $\lam_\Kc$ is
a geo-lamination. We are especially interested in collections of angles which
give rise, through the above construction, to specific geo-laminations
reminiscent of the case (3) of Theorem~\ref{neces}.

\begin{defn}\label{def-siegel}
Suppose that $\Hc$ is a collection of finite sets of angles $H_i, i=1, \dots,
m$ such that the following holds.

\begin{enumerate}

\item Each set $H_i$ is mapped into a one-angle set (i.e., is an \emph{all-critical set}).

\item For each $i$ the set $\imp(H_i)$ is a continuum disjoint
from impressions of any angle not belonging to $H_i$.

\item The continua $\imp(H_i), i=1, \dots, m$ form a wandering
collection.

\item Consider the geo-lamination $\ol{\lam_{\Hc}}$. Then
there is a cycle of Siegel domains in $\ol{\lam_{\Hc}}$ such that $\Hc$ is the
family of all-critical gaps/leaves on the boundaries of domains from the cycle.
Moreover, each $\ch(H_i)$ meets the corresponding Siegel domain of
$\ol{\lam_{\Hc}}$ in a leaf and sets $H_i$ have pairwise disjoint orbits.

\end{enumerate}

In that case we say that the collection of sets of angles $\Hc$ with their
impressions and \emph{all} their pullbacks form a \emph{Siegel configuration};
the collection of sets of angles $\Hc$ is said to \emph{generate} the
corresponding Siegel configuration. We will also say in this case that $P$
\emph{admits a Siegel configuration.}
\end{defn}

The next proposition shows that such Siegel configuration
cannot be admitted by the polynomial inside a periodic infinite K-class; in
particular, if $P$ admits a Siegel configuration, it implies that the finest
model is non-degenerate.

\begin{prop}\label{siegel}
Suppose that $p\in J_{\sim_P}=\ph(J_P)$ is a periodic point such that
$\Phi^{-1}(p)=\g$ is infinite. Then no collection of subsets of $\g$ generates
a Siegel configuration. In particular, if $P$ admits a Siegel configuration,
then the finest model is non-degenerate.
\end{prop}

\begin{proof}
By way of contradiction let us assume that $P$ admits a Siegel configuration,
and the corresponding generating collection of sets of angles is $\Hc=\{H_1,
\dots, H_m\}$. Set $\imp(H_i)=T_i$. First we simply analyze the corollaries of
this assumption without assuming that sets from $\Hc$ are contained in a
periodic K-class.


We may assume that all sets $H_i$ have common leaves with an \emph{invariant}
Siegel domain $S$. By Lemma~\ref{good-gap} the map $\si^*|_{\bd(S)}$ is
semiconjugate with an irrational rotation of the circle. Then there are no
periodic leaves/points in $\bd(S)$ and by Lemma~\ref{lem:periodic_gap} every
leaf $\ell\subset \bd(S)$ is (pre)critical. By Lemma~\ref{no-crit-leaf} $\ell$
is not a limit leaf, hence $\ell$ belongs to an element $Q$ of the grand orbit
of $\Hc$. From part (4) of Definition~\ref{def-siegel} $Q\cap \bd(S)=\ell$.
Since grand orbits of sets $H_i$ are pairwise disjoint, all images of $\ell$
are two-sided limit points of $\bd(S)\cap \uc$. Observe that there might exist
chains of concatenated leaves in $\bd(S)$ (they may arise as a result of
pulling back a set $H_i$ through a critical gap on the boundary of $S$).
However by Lemma~\ref{good-gap} any maximal chain of leaves in $\bd(S)$
consists of no more than $N$ leaves with some uniform $N$. Points of $\bd(S)$
which are not contained in any leaf are angles whose impressions are also
continua. Let us denote by $\A$ the collection of elements of the grand orbit
of $\Hc$ non-disjoint from $\bd(S)$ as well as points in $\bd(S)$ which do not
belong to leaves. Then all elements of $\A$ have connected impressions.


Suppose now that $A, B\in \A$. Choose the arc $I\subset \bd(S)$ which contains
$A\cap \bd(S), B\cap \bd(S)$ and runs in a counterclockwise direction from
$A\cap \bd(S)$ to $B\cap \bd(S)$. Consider the union $T=T(A, B)$ of all
elements of $\A$ non-disjoint from $I$. Clearly, $T$ is connected.

\smallskip

\noindent \textbf{Claim A.} \emph{The set $\imp(T)$ is a
continuum.}

\smallskip

It follows from the upper semi-continuity of impressions that $\imp(T)$ is
closed. By way of contradiction suppose that $\imp(T)=X\cup Y$ where $X, Y$ are
disjoint non-empty closed sets. Since for every $Q\in \A$ such that $Q\subset
T$ we have that the set $\imp(Q)$ is a continuum, every such $Q$ has its
impression either in $X$ or in $Y$. Denote by $X'$ the union of all elements of
$\A$ contained in $T$ whose impressions are contained in $X$; then $X'$ is
well-defined and disjoint from the union $Y'$ of all elements of $\A$ contained
in $T$ whose impressions are contained in $Y$. Now, by the upper
semi-continuity of impressions the sets $X', Y'$ are closed (every limit set of
$X'$ still comes from $T$ and has its impression in $X$), and by the above they
are disjoint and non-empty. However $X'\cup Y'=T$ is connected, a
contradiction. This implies that $\imp(\A)$ is a continuum.

\smallskip

\noindent \textbf{Claim B.} \emph{Impressions of two distinct
elements $A, B$ of $\A$ do not meet. The continuum $\imp(\A)$
separates the plane.}

\smallskip

Indeed, suppose otherwise. Choose a set $H_1\in \Hc$. Then $H_1\cap \bd(S)$ is
a leaf. By Lemma~\ref{good-gap}, there exists a sequence $m_i$ such that
$\si^{m_i}(A)$ will approach an endpoint of $H_1\cap \bd(S)$ while
$\si^{m_i}(B)$ will approach a point $y\in S'$. Now, $y\nin H_1$ because $A$ is
distinct from $B$ and because the map $\si$ on $\bd(S)$ acts like an irrational
rotation. On the other hand, by the assumption $\imp(A)\cap \imp(B)\ne \0$,
hence by the upper semi-continuity of impressions $\imp(y)\cap \imp(H_1)\ne 0$,
a contradiction with the part (2) of Definition~\ref{def-siegel}. Hence
elements of $\A$ have pairwise disjoint impressions. It implies that $\imp(\A)$
separates the plane because otherwise by \cite{akis99, bfmot10} $\imp(\A)$ would
contain a fixed point, and then an element of $\A$ containing it and its image
would have non-disjoint impressions, a contradiction.


\smallskip

\noindent \textbf{Claim C.} \emph{The union of two impressions of distinct
angles - images of elements of $\Hc$ - ray-separates $\imp(\A)$.}

\smallskip

Consider $\{\al\}, \{\be\}\in \A, \al\ne \be$, both $\al$ and $\be$ images of
sets from $\Hc$ which are non-isolated from either side in $\bd(S)\cap \uc$ (we
can do this by what we showed in the second paragraph of the proof). We need to
show that if $Q=\imp(\al)\cup \imp(\be)$ then $\imp(\A)$ meets two distinct
components of $\ol{U_\iy}\sm \widetilde Q$ ($U_\iy$ is the basin of attraction
of infinity, $Q$ is a ray-compactum with the associated set of angles $\{\al,
\be\}$, and by $\widetilde Q$ we denote the union of $Q$ and rays $R_\al,
R_\be$, see Section 3 where this notation is introduced). Consider the union
$V$ of rays of all angles from $[\al, \be]$ and the union $W$ of rays of all
angles from $[\be, \al]$. Clearly, $V\cap W=R_\al\cup R_\be$ and $V\cup
W=U_\iy$. Also, it follows that $\ol{V}=V\cup \imp([\al, \be])$ and
$\ol{W}=W\cup \imp([\be, \al])$.

Let us show that $\ol{V}\cap \ol{W}=\widetilde Q$. It suffices to show that if
$T'=\imp([\al, \be])\sm Q$ and $T''=\imp([\be, \al])\sm Q$ then $T'\cap
T''=\0$. Observe that by Claim B and by the choice of $\al, \be$ we  have that
$\imp(\al)$ is disjoint from impressions of all angles not equal to $\al$, and
$\imp(\be)$ is disjoint from impressions of all angles not equal to $\be$.
Hence it suffices to show that if $\ga'\in (\al, \be)$ and $\ga''\in (\be,
\al)$ then $\imp(\ga')\cap \imp(\ga'')=\0$. By Claim B we may assume that at
least one of the angles $\ga', \ga''$ (say, $\ga'$) does not belong to an
element of $\A$. Then there exists a non-degenerate element $L$ of $\A$ such
that $L\cap \uc\subset (\al, \be)$ and $\ga'$ is contained in an arc $(\ta_1,
\ta_2)\subset (\al, \be)$ where $\ta_1, \ta_2\in L$. This implies that
$\imp(\ga')$ is contained in the union of rays $R_{\ta_1}, R_{\ta_2}$ and the
impression $\imp(L)$ of $L$. If $\ga''$ belongs to $H\in \A$, put $M=H$ and
$\ta_3=\ta_4=\ga''$. Otherwise, there exists a set $M\in \A$ such that $M\cap
\uc\subset (\be, \al)$ and $\ga''$ is contained in an arc $(\ta_3,
\ta_4)\subset (\be, \ga)$ where $\ta_3, \ta_4\in L$. Then $\imp(\ga'')$ is
contained in the union of rays $R_{\ta_3}, R_{\ta_4}$ and the impression
$\imp(M)$ of $M$. Since by Claim B $\imp(L)\cap \imp(M)=\0$, it follows that
$\imp(\ga')\cap \imp(\ga'')=\0$ as desired. Observe that $\ol{U_\iy}\sm
\widetilde Q=(\ol{V}\sm \widetilde Q)\cup (\ol{W}\sm \widetilde Q)$ where sets
$\ol{V}\sm \widetilde Q$ and $\ol{W}\sm \widetilde Q$ are open in $\ol{U_\iy}$ and
disjoint which proves the claim.

Let us now prove the theorem. Observe that by Claim A the set $\imp(A)$ is a
continuum. Denote by $\Z$ the family of impressions of singletons from $\A$
which are angles-images of elements of $\Hc$. By Lemma~\ref{good-gap} the map
$\psi$ semiconjugates $\si^*|_{\bd(S)}$ to an irrational rotation $\tau$ of
$\uc$. This map allows us to associate to elements of $\Z$ their $\psi$-images
which are angles in $\uc$ coming from a finite collection of orbits under
$\tau$. Choose pairs of angles from $\psi(\Z)$ so that $\uc$ with them is
homeomorphic to $\uc$ with the vertical collection of pairs $\mathcal C_{\uc}$.
This gives rise to the corresponding family of pairs of impression from $\Z$.
By Claim C and by the construction these pairs of impressions form a
well-slicing family of $\imp(\A)$. Therefore by Theorem~\ref{th:wc-non-deg}
$\ph(\imp(A))$ is not a point. On the other hand, by the construction
$\imp(\A)\subset \ph^{-1}(p)$, a contradiction.
\end{proof}

\subsection{The criterion}

First we deal with parattracting Fatou domains. This sufficient
condition for the non-collapse of a subset of $J_P$ corresponds to
case (1) of Theorem~\ref{neces}. Let us recall that by $R$ we denote
the set of all periodic repelling (parabolic) \emph{bi-accessible}
points and their preimages.

\begin{prop}\label{fatliv1}
  Suppose that $U$ is parattracting Fatou domain of $P$.  Then $\bd(U)$ is
  well-sliced in $J_P$ and hence is not collapsed under the finest map
  $\ph$. In particular, suppose that $p\in J_{\sim_P}$ is a periodic
  point.  Then $\ph^{-1}(p)$ cannot contain the boundary of a
  parattracting Fatou domain of $P$.
\end{prop}

\begin{proof}
  By \cite{pz94}, $R \cap \bd(U)=A$ is dense in $\bd(U)\subset X$ and each
  point of $A$ is accessible from within and from without $U$. This implies
  that any pair of points of $A$ ray-separates $\bd(U)$. Since $A$ consists of
  points accessible from within $U$ we can use the canonic Riemann map for $U$
  and parameterize points of $A$ by the corresponding angles; denote the
  corresponding set of angles by $\A$. Since all points of $A$ are accessible
  from outside $U$ and $A$ is dense in $\bd(U)$, it follows that $\A$ is dense
  in $\uc$. Since $R$ is countable, so is $\A$, and it is easy to see that we
  can choose pairwise disjoint pairs of angles from $A$ so that $\uc$ with this
  collection of pairs is homeomorphic to $\uc$ with the vertical collection of
  pairs $\mathcal C_{\uc}$ defined in the end of Section~\ref{model}. Then the
  corresponding to these pairs of angles pairs of points from $A$ form a
  well-slicing family of $\bd(U)$ and by Theorem~\ref{th:wc-non-deg} $\bd(U)$
  is not collapsed under the finest $\ph$ as desired.
\end{proof}

We are ready to state the main result of this section which gives a criterion
of the finest model not be degenerate. It lists three conditions, and for the
finest model to be non-degenerate  it is necessary and sufficient that at least
one of them must be satisfied. In a descriptive form it was given in
Section~\ref{intro}.

\begin{thm}\label{csgap} The finest model of the Julia set of
a polynomial $P$ is not degenerate if and only if at least one of the following
properties is satisfied.

\begin{enumerate}

\item The filled-in Julia set $K_P$ contains a parattracting Fatou
  domain.

\item The set of all repelling bi-accessible periodic points is infinite.

\item The polynomial $P$ admits a Siegel configuration.

\end{enumerate}

\end{thm}

\begin{proof} First we show that the fact that at least one of
properties (1) - (3) holds is necessary for the non-degeneracy of
$J_{\sim_P}=\ph(J_P)$. In other words, we assume that $J_{\sim_P}$ is non-degenerate and
deduce the appropriate properties of $J_P$ using Theorem~\ref{neces}. Consider
the cases (1) - (3) one by one.

(1) Suppose that, according to Theorem~\ref{neces}.(1), $J_{\sim_P}$
contains a simple closed curve $S$ which is the boundary of a parattracting Fatou
domain.  Then $\ph^{-1}(S)$ is a continuum which
separates the plane and encloses an open set $U$ complementary to
$J_P$. Moreover, for a dense in $S$ subset of $g$-periodic points
their $\Phi$-preimages are finite (there are no more than finitely
many periodic points of $\ph(J_P)$ whose $\Phi$-preimages are
infinite). By Lemma~\ref{lcpt} this implies that full $\ph$-preimages
of these $g$-periodic points are $P$-periodic points at which $J_P$ is
locally connected. Thus, $U$ is a Fatou domain of $P$ whose boundary
contains periodic points. This implies that $U$ is a parattracting
domain, and case (1) holds.

(2) Assume now that $J_{\sim_P}$ does not contain simple closed curves, that is, that
$J_{\sim_P}$ is a dendrite. Consider the lamination $\sim_P$. Since $J_{\sim_P}$ is a
dendrite, $\sim_P$ does not have Fatou domains. Hence by Theorem~\ref{neces}
there are infinitely many periodic $\sim_P$-classes each of which consists of
more than one point. Moreover, we may assume that they are all finite (because
there can only be finitely many infinite periodic classes of a lamination).
Finally, by the construction the impression of each such class is disjoint from
impressions of all angles not belonging to the class. Hence by Lemma~\ref{lcpt}
all their impressions are points. We conclude that there are infinitely many
repelling bi-accessible periodic points as desired and case (2) holds.

(3) By Lemma~\ref{two-dyns} we may now assume that $\ph(J_P)$ contains the
boundary $S$ of an invariant Siegel domain. By Theorem~\ref{neces}.(3), there
exists a finite collection of all-critical $\sim_P$-classes $\Hc=\{H_1, \dots,
H_m\}$ with pairwise disjoint grand orbits whose images $x_1, \dots, x_m$ under
the quotient map $\Phi:\uc\to \ph(J_P)=J_{\sim_P}$ form the set of all-critical
points in $S$ so that \emph{all} cutpoints of $\ph(J_P)$ in $S$ belong to the
grand orbits of these all-critical points. Observe that by the construction for
every $i$ we have that $\imp(H_i)=\ph^{-1}(x_i)$ is a continuum.

We want to show that this implies that $P$ admits a Siegel configuration. As
the collection of sets of angles needed to define a Siegel configuration we
take exactly $\Hc$. Moreover, as in the definition of a Siegel configuration we
take the grand orbit of $\Hc$ then the corresponding sets of angles to form the
geometric prelamination $\lam_{\Hc}$. Observe that this will bring back
\emph{all} the leaves and gaps from the set $\Phi^{-1}(S)$ because all leaves
and gaps in this set correspond to cutpoints of $J_{\sim_P}$ in $S$ and, by
Theorem~\ref{neces}.(3), come from the grand orbits of all-critical points from
$S$. Finally, by the construction the impressions $\imp(H_i)$ are disjoint from
impressions of all angles not belonging to $H_i$. All this implies that $P$
admits a Siegel configuration and completes the consideration of the case (3).

Now we consider the sufficiency of conditions (1) - (3). If (1) holds, then the
finest map is not degenerate by Proposition~\ref{fatliv1}. If (2) holds, then
the finest map is not degenerate by Proposition~\ref{biacc}. If (3) holds, then
the finest map is not degenerate by Proposition~\ref{siegel}. This completes
the proof.
\end{proof}

By Theorem~\ref{csgap} the finest model of a polynomial Julia set is degenerate
if and only if there are no parattracting Fatou domains, the set of all
repelling bi-accessible periodic points is finite, and there is no Siegel
configuration. As an application let us first prove a sufficient condition for
the finest model to be non-degenerate. Recall that the valence of a ray-continuum $K$ is
the cardinality of the set of all rays whose principal sets are contained in $K$.

\begin{thm}\label{suff1}
Suppose that $K'$ is a wandering ray-continuum such that the valence
of $P^n(K')$ is greater than $1$ for all $n\ge 0$. Then there are
infinitely many repelling bi-accessible periodic points of $J$ and
the finest model is non-degenerate. In particular, these conclusions
hold if there exists a bi-accessible point of $J$ which is
non-(pre)periodic and non-(pre)critical.
\end{thm}

\begin{proof}
As explained in Subsection 5.3, the construction and the arguments similar to
those from Theorem 4.2 \cite{bo08a} imply that there is a (possibly) bigger
than $K'$ but still wandering ray-continuum $K$ (with the same eventual images
as $K'$) whose grand orbit $\Ga$ (i.e. the collection of pullbacks of its
forward images) is well-defined. Moreover, to each element $Q$ of $\Ga$ we can
associate the set $\Ta(Q)=H_Q$ of all angles whose principal sets are contained
in $Q$ (by Theorem~\ref{kiwi-wan} the set $H_Q$ is finite). Then all elements
of $\Ga$ are non-separating and wandering ray-continua. Moreover, convex hulls
of sets $H_Q, Q\in \Ga$ form a prelamination which we denote $\lam_K$. By the
properties

By Theorem~\ref{periodic-prel} we can consider its closure
$\ol{\lam_K}$ which is the geo-lamination generated by $K$ and then
the lamination $\approx_K$ generated by $K$. By
Theorem~\ref{periodic-prel} $\approx_K=\approx$ has no Siegel
domains. However it may have several parattracting Fatou domains.

Let us show that closures of Fatou domains of $\approx$ are pairwise disjoint.
Let $U$ be a Fatou domain of $\ol{\lam_K}$. By the construction from
Theorem~\ref{periodic-prel}, $U$ remains a Fatou domain of $\approx$. Let us
study $\bd(U)$ in detail. By Theorem~\ref{periodic-prel} in the geo-lamination
$\ol{\lam_K}$ and in the refined geo-lamination $\lam_{\approx}$ there are no
critical leaves. Therefore by Lemma~\ref{lem:periodic_gap} all leaves in
$\bd(U)$ are (pre)periodic. Thus, they do not come from $\lam_K$ and must be
the limit leaves of $\lam_K$. Choose a geometric leaf $\ell$ in $\bd(U)$. By
Theorem~\ref{periodic-prel} elements of $\lam_K$ cannot be contained in
$\ol{U}$, hence they approach $\ell$ from outside of $U$. Moreover, we may
assume that these elements of $\lam_K$ are contained in convex hulls of
distinct $\approx$-classes. Therefore $\ell$ cannot lie on the boundary of any
other gap of $\lam_{\approx}$ or on the boundary of another Fatou domain of
$\ol{\lam_K}$ (or, equivalently, of $\approx$), as desired.

Consider a new lamination $\approx'_K=\approx'$ obtained by identifying the
boundary of each Fatou domain of $\approx_K$ and show that $J_{\approx'}$ is
a non-degenerate dendrite. It is easy to see that $\approx'$ is a well-defined
lamination. Then the corresponding topological Julia set $J_{\approx'}$ can be
obtained from $J_\approx$ by collapsing closures of all its Fatou domains into
points. Clearly, there are no more than countably many Fatou domains of
$\approx$, their boundaries are continua, and these continua are pairwise
disjoint by the previous paragraph. Then by the Sierpi\'nski Theorem
\cite{sie18} the resulting (after this collapse) quotient space $J_{\approx'}$
is not degenerate. Hence the lamination $\approx'$ is not degenerate. Moreover,
since it no longer has Fatou domains, $J_{\approx'}$ is a dendrite.

By Theorem 7.2.7 of \cite{bfmot10} any dendritic topological Julia
set has infinitely many periodic cutpoints. Hence there are
infinitely many periodic cutpoints in $J_{\approx'}$. We now want to
show that this implies that there are infinitely many periodic
cutpoints of $J$. Let $\h$ be a finite periodic class of $\approx'$
which does not belong to the boundary of a Fatou domain of
$\approx$. Then geometric leaves from $\bd(\ch(\h))$ cannot come
from elements of $\lam_K$ (who are all wandering). Let us show that
all geometric leaves on the boundary of $\h$ are limit leaves of
$\lam_K$. Indeed, suppose that $\ell'$ is a boundary geometric leaf
of $\ch(\h)$ which is not such a limit leaf. Then there is a
geometric gap $\g'$ of $\ol{\lam_K}$ on the side of $\ell'$ opposite
to $\g$. By the choice of $\h$, the gap $\g'$ cannot be a Fatou
domain of $\ol{\lam_K}$ which implies that it has a finite basis
which should have been united with $\h$ into one $\approx$-class, a
contradiction. Thus, the set $\imp(\h)$ is disjoint from impressions
of all angles not in $\h$ because these other impressions are cut
off $\imp(\h)$ by the ray-continua from the grand orbit of $K$
corresponding to the appropriate elements of $\lam_K$.

Consider now the set $\imp(\h)$ and show that $\imp(\h)$ is a continuum itself.
If a geometric leaf $\ell''$ belongs to the boundary of $\ch(\h)$ then by the
previous paragraph $\ell''$ is the limit of a sequence of elements of
$\lam_\Ta$. Taking the Hausdorff limit of a subsequence we see that the
corresponding continua on the plane converge to a continuum. By the
semi-continuity of impressions this continuum is contained in $\imp(\ell'')$.
Hence $\imp(\ell'')$ is a continuum itself. Since the union of impressions of
leaves $\ell''$ from the boundary of $\ch(\h)$ is in fact $\imp(\h)$, the set
$\imp(\h)$ is a continuum. By Lemma~\ref{lcpt} $\imp(\h)$ is a repelling or
parabolic periodic point, and since $\h$ is a gap or leaf, it is a repelling or
parabolic point of $J$ at which at least two rays land, as desired. By
Theorem~\ref{csgap} this implies that the finest model is non-degenerate.
Clearly, the case when there exists a non-(pre)periodic bi-accessible point of
$J$ is a particular case of the above. This completes the proof.
\end{proof}

Let us show how one can deduce Kiwi's results \cite{kiwi97} from our results.
Say that two angles $\al, \be$ are \emph{K-equivalent} if there exists a finite
collection of angles $\al_0=\al, \dots, \al_k=\be$ such that $\imp(\al_i\cap
\imp(\al_{i+1})\ne \0$ for each $i=0, \dots, k-1$. The notion (but not the
terminology!) is due to Jan Kiwi \cite{kiwi97} and is instrumental in his
construction of locally connected models for connected Julia sets of
polynomials without CS-points. Clearly, if two angles are K-equivalent, they
must belong to the same K-class. Suppose that $P$ does not have CS-points. Let
us show first that the finest model is non-degenerate. Indeed, by the
assumption $P$ has no Siegel domains. If $P$ has a parattracting domain then by
Theorem~\ref{csgap} the finest model is non-degenerate. It remains to consider
the case when $P$ has no Fatou domains (i.e., $J_P$ is non-separating) and no
CS-points. Then by \cite{gm93,kiwi00} $P$ has infinitely many repelling
periodic bi-accessible points. Hence in this case the finest model is
non-degenerate either.

Now, take any point $p$ of $P$, consider the corresponding K-class
$\Phi^{-1}(p)$ and show that it is finite. Indeed, suppose first that
$p$ is non-(pre)periodic. Then by Theorem~\ref{kiwi-wan} the
corresponding K-class is finite. Now suppose that $p$ is (pre)periodic; we
may assume that it is periodic of period $1$. Consider the set
$Q=\ph^{-1}(\ph(p))$ and show that it is non-separating.  Indeed,
otherwise there is a parattracting domain $U$ contained in the
topological hull $\tl(Q)$ (since $P$ does not have CS-points it cannot
be a Siegel domain). However by Lemma~\ref{fatliv1} the boundary
$\bd(U)$ is not collapsed under $\ph$, a contradiction. Hence $Q$ is
non-separating. Let us show that then it must contain infinitely many
repelling periodic bi-accessible points. Indeed, suppose
otherwise. Then replacing $P$ by an appropriate power we may assume
that all periodic points in $Q$ and all the rays landing at them are
invariant. By Theorem~\ref{degimp} this implies, that $Q$ is a point,
a contradiction to $\Phi^{-1}(p)$ being infinite by the
assumption (at any repelling periodic point only finitely many
rays land). So, if $P$ has no CS-points then there are no infinite
K-classes which implies that K-equivalence in fact coincides with the
lamination $\sim_P$ and thus produces the finest locally connected
model of $J_P$.

Let us compare our approach and results with those of \cite{kiwi97}. Kiwi uses
direct arguments to construct the finest model for polynomials without
CS-points. He also relies more upon combinatorial and related to symbolic
dynamics arguments. Our approach, based upon continuum theory, is different. It
allows us to show that Kiwi's locally connected model of a connected Julia set
without CS-points is actually the \emph{finest locally connected model of
$J_P$}, the finest from the purely topological point of view. It also allows us
to extend Kiwi's results \cite{kiwi97} onto \emph{all} polynomials with
connected Julia sets. However we only tackle the case of connected Julia sets
while in \cite{kiwi97} disconnected Julia sets are also considered.

To conclude the paper we want to specify K-equivalence a little more. Namely,
in the next theorem we obtain additional information about the way impressions
of angles from finite K-classes can intersect. The theorem holds regardless of
whether a polynomial has CS-points or not. However in the case when $P$ has no
CS-points it applies to \emph{all} K-classes.

\begin{thm}\label{ptn} Suppose that $A=\{\al_1, \dots, \al_n\}$ is a
finite K-class. Then impressions of angles of $A$ which are adjacent on the circle meet.
Moreover, any subset of $A$ in which \emph{only} adjacent angles have meeting
impressions consists of no more than 3 elements.
\end{thm}

\begin{proof} In the case when $A$ is a (pre)periodic K-class (equivalently,
$\sim_P$-class) it follows from Lemma~\ref{lcpt} that $\imp(A)$ is a point
which implies the conclusions of the lemma. Also, if $A$ consists of two angles
the conclusions of the lemma are obvious. Hence the remaining case is when
$n\ge 3$ and $A$ is a wandering polygon. Consider this case
by way of contradiction. Assume that $\al_1, \dots, \al_n$ circularly ordered
and $\imp(\al_1)\cap \imp(\al_2)=\0$. Denote the open arc between $\al_1,
\al_2$ which is complementary to $A$ by $I$.

Let us show that there exists a Fatou domain $U$ and a point of $x\in
[\imp(A)\sm (\imp(\al_1)\cup \imp(\al_n))]\cap \bd(U)$. Draw a curve $L$ which
starts at a point of a ray of an angle from $I$ and ends at a point of a ray of
an angle from $\uc\sm \ol{I}$. Clearly, $L$ separates $\imp(\al_1)$ from
$\imp(\al_2)$. Since $\imp(A)$ is a continuum, $L$ will have to intersect
$\imp(A)$. Denote by $x$ the first on $L$ point of intersection between $L$ and
$\imp(A)$. Let us show that a sufficiently small open subarc $T$ of $L$ with
one endpoint $x$ and disjoint from $\imp(A)$ is in fact disjoint from $J_P$.
Indeed, since $\al_1$ and $\al_n$ are adjacent elements of $A$, the set
$\cup_{\ga \in I}\imp(\ga)$ is disjoint from $\imp(A)$, and hence does not
contain $x$. On the other hand, $x\nin \imp(\al_1)\cup \imp(\al_2)$ by the
choice of $L$. Hence $x\nin \cup_{\ga \in \ol{I}}\imp(\ga)=Q$, and since $Q$ is
compact, we can find the desired arc $T$. On the other hand, the intersection
$\imp(A)\cap Q=\imp(\al_1)\cup \imp(\al_n)$ is disconnected which implies that
$Q$ separates the plane. By the construction $T$ must be contained in a bounded
component $U$ of $\C\sm Q$. Since $Q\subset J_P$, it follows that $U$ is a
Fatou domain, and hence $x\in \bd(U)$.

Take a small ball $B$ centered at $x$. By \cite{pz94} there exists a
(pre)periodic point $y\in B\cap \bd(U)$. Also, choose a (pre)periodic point
$y'\in \bd(U)$ so that a ray of an angle belonging to $I$ lands at $Y'$. Since
$\imp(A)$ is wandering, $y, y'\nin \imp(A)$. As in the proof of
Lemma~\ref{fatliv1}, connect a point $z\in U$ with infinity by a curve $E$
which intersects $J_P$ only at $y$ and $y'$. Then $L'$ separates $\imp(\al_1)$
from $\imp(\al_2)$ on the plane and is disjoint from the continuum $\imp(A)$
which contains both $\imp(\al_1)$ and $\imp(\al_n)$, a contradiction. Thus,
adjacent angles in $A$ must have non-disjoint impressions.

To prove the rest, assume that there exist angles $\be_1, \dots, \be_r\in A,
r\ge 4$ which are circularly ordered and such that all adjacent angles have
non-disjoint impressions while otherwise the impressions of angles are
disjoint. Consider two continua, $Y=\imp(\be_1)\cup \imp(\be_2)$ and
$Z=\cup^r_{i=3}\imp(\be_i)$. Then it follows that

$$Y\cap Z=[\imp(\be_1)\cap \imp(\be_r)]\cup [\imp(\be_2)\cap \imp(\be_3)]$$

\noindent which is disconnected because $\imp(\be_1)\cap \imp(\be_3)=\0$
(recall that $r>3$). Hence $\imp(A)$ separates the plane which is impossible.
Indeed, if $\imp(A)$ separates the plane then its topological hull contains a
Fatou domain and $\imp(A)$ is (pre)periodic. Assume that $\imp(A)$ (and $A$) are
periodic of period $1$. If $\imp(A)$ contains the boundary of an parattracting
Fatou domain then by \cite{pz94} $\imp(A)$ will have to intersect infinitely
many impressions, a contradiction. If $\imp(A)$ contains the boundary of a
Siegel domain then by Lemma~\ref{hedgehog} it contains a critical point $c\in
J_P$ and $A$ contains at least two angles with the same $\si$-image. However,
as $A$ is a finite invariant K-class, the map $\si$ maps $A$ onto itself in a
one-to-one fashion, a contradiction.
\end{proof}


\begin{thebibliography}{00}




\bibitem{akis99} V. Akis, \emph{On the plane fixed point problem}, Topology
Proc., \textbf{24} (1999), 15--31.

\bibitem{bcmo08} A. Blokh, D. Childers, J. Mayer, L. Oversteegen,
\emph{Non-degeneracy of laminations}, preprint, arXiv:0809.2019 (2008)

\bibitem{bl02} A. Blokh and  G. Levin, \emph{An inequality for laminations,
Julia sets and `growing trees'}, Erg. Th. and Dyn. Sys.,
\textbf{22} (2002), pp. 63--97.

\bibitem{bo06} A. Blokh and  L. Oversteegen, \emph{The Julia sets of
quadratic Cremer polynomials}, Topology and its Applications,
\textbf{153} (2006), pp. 3038--3050.

\bibitem{bo06b} A. Blokh and  L. Oversteegen, \emph{Monotone images
of Cremer Julia sets}, Houston Journal of Mathematics, \textbf{36}
(2010), pp. 469--476.

\bibitem{bo08a} A. Blokh, L. Oversteegen, \emph{The Julia sets of basic
uniCremer polynomials of arbitrary degree}, Conformal Geometry and
Dynamics, \textbf{13} (2009), pp. 139--159.

\bibitem{bfmot10} A. Blokh, R. Fokkink, J. Mayer, L. Oversteegen,
E. D. Tymchatyn, \emph{Fixed point theorems in plane continua with
applications}, preprint, arXiv:1004.0214 (2010).

\bibitem{dou93} A. Douady, \emph{Descriptions of compact sets in $\C$},
Topological Methods in Modern Mathematics, Publish or Perish (1993), pp.
429--465.

\bibitem{douahubb85} A. Douady and  J. H. Hubbard,  \emph{\'Etude
dynamique des polyn\^omes complexes I, II} Publications
Math\'ematiques d'Orsay \textbf{84-02} (1984), \textbf{85-04}
(1985).

\bibitem{gm93} L.~Goldberg and  J.~Milnor, \emph{Fixed points of
polynomial maps, Part II. Fixed point portraits}, Ann. Sci. Ec.
Norm. Sup., $4^e$ s\'erie \textbf{26} (1993), pp. 51--98.

\bibitem{kiwi00} J. Kiwi, \emph{Non-accessible critical points of Cremer
polynomials}, Erg. Theory and Dyn. Sys. \textbf{20}, (2000), pp.
1391--1403.

\bibitem{kiwi02} J. Kiwi, \emph{Wandering orbit portraits}, Trans. Amer.
Math. Soc. \textbf{254} (2002), pp. 1473--1485.

\bibitem{kiwi97} J. Kiwi, \emph{$\mathbb R$eal laminations and the topological
dynamics of complex polynomials}, Advances in Math. \textbf{184}
(2004), no. 2, pp. 207--267.

\bibitem{Melo:1993nx}
W.~de~Melo and S.~van Strien.
\newblock {\em One-dimensional dynamics}, \textbf{25}, {\em Ergebnisse der
Mathematik und ihrer Grenzgebiete (3) [Results in Mathematics and Related
Areas (3)]}.
\newblock Springer-Verlag, Berlin, 1993.

\bibitem{Milnor:2006fr} J. Milnor,
\newblock {\em Dynamics in one complex variable}, \textbf{160},
\newblock Princeton University Press, Princeton, 3rd ed edition, 2006.

\bibitem{mt} J. Milnor, W. Thurston, \emph{On iterated maps of the interval},
in: Dynamical systems (College Park, MD, 1986--87), Lecture Notes in Math.,
\textbf{1342}, Springer-Verlag, Berlin, 1988.

\bibitem{mr07} M. Misiurewicz, A. Rodrigues, \emph{Double standard maps},
Commun. Math. Phys. \textbf{273} (2007), 37-65.

\bibitem{m25} R. L. Moore, \emph{Concerning upper semi-continuous
collection of continua}, Trans. Amer. Math. Soc., \textbf{27}
(1925), 416--428

\bibitem{nad92} S. Nadler, Jr, \emph{Continuum theory. An introduction.}
Monographs and Textbooks in Pure and Applied Mathematics
\textbf{158}, Marcel Dekker, Inc., New York (1992).

\bibitem{pere94} R. Perez-Marco, \emph{Topology of Julia sets
and hedgehogs}, Publications Math\'ematiques d'Orsay
\textbf{94-48} (1994).

\bibitem{pere97} R. Perez-Marco, \emph{Fixed points and circle
maps}, Acta Math. \textbf{179} (1997), pp. 243--294.

\bibitem{pz94} F. Przytycki, A. Zdunik, \emph{Density of periodic
sources in the boundary of a basin of attraction for iteration of
holomorphic maps: geometric coding trees techniques}, Fund. Math.
\textbf{145}, no. 1, 65--77.

\bibitem{sie18} W. Sierpi\'nski, \emph{Un th\'eor\'eme sur les continus},
T\^{o}hoku Math. J., \textbf{13} (1918), pp. 300--303.


\bibitem{sto56} S. Stoilow, \emph{Le\c{c}ons sur les principes topologiques
de la th\'eorie des fonctions analytiques}, 2i\`eme ed. (1956),
p. 121, Gauthier-Villars, Paris

\bibitem{thur85}
W.~P. Thurston, \emph{On the geometry and dynamics of iterated
rational maps}, in: ``Complex dynamics: families and friends'', A K
Peters (2008), pp. 1--108.


\end{thebibliography}
\end{document}